\documentclass[12pt,a4paper]{amsart}
\usepackage{graphicx}
\usepackage[T1]{fontenc}
\usepackage{lmodern}
\usepackage{exscale}
\usepackage{amssymb}
\usepackage{calrsfs}  
\usepackage{tikz}

\numberwithin{equation}{section}
\numberwithin{figure}{section}

\newtheorem{thm}{Theorem}
  
\newtheorem{thrm}{Theorem}[section]
\newtheorem*{theo}{Theorem}
\newtheorem{prop}[thrm]{Proposition}

\newtheorem{cor}[thrm]{Corollary}

\theoremstyle{definition}
 \newtheorem*{dfn}{Definition}
 \newtheorem{df}[thrm]{Definition}

 \newtheorem*{exs}{Examples}

\theoremstyle{remark}
 \newtheorem*{rmk}{Remark}

\renewcommand{\phi}{\varphi}
\renewcommand{\epsilon}{\varepsilon}
\renewcommand{\setminus}{\smallsetminus}
\newcommand{\rgt}[1]{\right#1}
\newcommand{\lft}[1]{\left#1}

\newcommand{\rd}{\partial}

\newcommand{\Z}{\mathbb{Z}}

\newcommand{\sect}[1]{\operatorname{Sec}\left(#1\right)}
\newcommand{\secc}[1]{\operatorname{Sec}#1}
\newcommand{\D}{\mathcal{D}}
\newcommand{\DD}{\underbar{$\D$}}
\newcommand{\ub}[1]{\underbar{$#1$}}
\newcommand{\rjet}[1]{{#1}^{(r)}}
\newcommand{\ojet}[1]{{#1}^{(1)}}
\newcommand{\vertt}{\operatorname{Vert}}
\newcommand{\symb}[1]{\operatorname{Symb}#1}
\newcommand{\dsec}[1]{\operatorname{Sec}_{\mathcal{F}}(#1)}
\newcommand{\rk}{\operatorname{rank}}
\newcommand{\almtf}[1][M]{\Omega^{\mathsf{alm}}_{(3,5)}(#1)}
\newcommand{\almcrtn}[1][M]{\Omega^{\mathsf{alm}}_{(2,3,5)}(#1)}
\newcommand{\zed}[3]{z^{#1}_{{#2}{#3}}}
\newcommand{\ba}[4]{B^{#1}_{#2}A_{{#3}{#4}}}
\newcommand{\aaa}[2]{{A_{#1 #2}}}
\newcommand{\cc}[2]{{C^{#1}_{#2}}}
\newcommand{\ca}[4]{\cc{#1}{#2}\aaa{#3}{#4}}
\newcommand{\bs}[1]{\boldsymbol{#1}}
\newcommand{\cv}[1]{\operatorname{Conv}\lft(#1\rgt)}
\newcommand{\spn}[1]{\operatorname{Span}\{#1\}}

\allowdisplaybreaks

\begin{document} 
\title[Cartan~$(2,3,5)$-distribution]
{Existence and classification \\ of the Cartan~$(2,3,5)$-distribution}
\author{Jiro ADACHI}
\thanks{This work was supported by JSPS KAKENHI Grant Number JP17K05236, 
  JP22K03305}
\address{Department of Mathematics, Hokkaido University, 
  Sapporo, 060--0810, Japan}
\email{j-adachi@math.sci.hokudai.ac.jp}


\begin{abstract}
  The Cartan $(2,3,5)$-distribution is a tangent distribution of rank~$2$ 
on a $5$-dimensional manifold satisfying certain generic conditions. 
 The necessary and sufficient condition for a manifold 
to admit such a structure is established in this paper. 
 The condition obtained is purely topological. 
 In addition, the classification of such structures, 
up to homotopy as formal Cartan distributions, is obtained. 
\end{abstract}

\maketitle

\section{Introduction}\label{sec:intro}
  The relation between topology of manifolds and 
geometric structures defined on them 
is an interesting and important problem in differential topology. 
 In addition, once a certain kind of structure is known to exist there, 
classification of such structures becomes a natural problem. 
 In this paper, we study such problems for the Cartan $(2,3,5)$-distribution. 

  A \emph{Cartan $(2,3,5)$-distribution}\/ 
is a certain tangent distribution $\D\subset TM$ of rank~$2$ 
on a $5$-dimensional manifold $M$. 
 It is defined as that satisfies the following conditions: 
\begin{equation*}
  \rk~[\D,\D]=3, \qquad \rk[\D,[\D,\D]]=5, 
\end{equation*}
where $[\D,\D]$ denotes the tangent distribution 
induced by the Lie bracket of vector fields 
(see Sections~\ref{sec:distr}, and~\ref{sec:Cdistr} for precise definitions). 
 Being introduced by Cartan in the famous ``$5$-variables paper''~\cite{Car5v}, 
it has been studied for a long time 
(see \cite{bcg3}, \cite{montgomery}, for example). 
 While it is related to control theory and various applied fields, 
it has always been a central subject of interest 
in the geometric study of differential equations. 

  One of the main goals of this paper is 
to give the necessary and sufficient condition 
for a manifold to admit a Cartan $(2,3,5)$-distribution. 
 We emphasize that the condition imposed on the manifold is purely topological. 
 In order to formulate the condition, we introduce some notions. 
 From the viewpoint of the \textit{h}-principle, 
we introduce a formal structure for the genuine Cartan~$(2,3,5)$-distribution. 
 An \emph{almost Cartan structure}\/ on a $5$-dimensional manifold is defined 
as a sequence of tangent distributions $\D\subset\mathcal{E}\subset TM$ 
of rank~$2$ and $3$ respectively 
together with a certain triple $(\omega_1,\ \omega_2,\ \omega_3)$ of $2$-forms 
(see Section~\ref{sec:almCrtn} for the precise definition). 
 The relation between the Cartan $(2,3,5)$-distribution 
and almost Cartan structure is analogous to that 
between contact structure and almost contact structure. 
 Then the necessary and sufficient condition is stated as follows. 
%
%
\begin{thm}\label{thma}
  Let $M$ be a possibly closed $5$-dimensional manifold. 
 Then $M$ admits a Cartan $(2,3,5)$-distribution 
if and only if it admits an almost Cartan structure. 
\end{thm} 
\noindent
 We should remark that the existence condition for an almost Cartan structure 
is purely topological. 
 Hence the existence problem for Cartan $(2,3,5)$-distributions 
reduces to a purely topological condition. 
 We prove this theorem from the viewpoint of the \textit{h}-principle. 
 Therefore if there exists an almost Cartan structure 
$(\D\subset\mathcal{E},\{\omega_1,\omega_2;\omega_3\})$, 
we have a genuine Cartan $(2,3,5)$-distribution $\tilde{\D}$ homotopic to $\D$. 
 The method applied to show the \textit{h}-principle in this paper 
is Gromov's convex integration method. 

  In some specific cases, similar results were obtained 
by Dave and Haller~\cite{davehaller}. 
 They dealt with the case where $5$-dimensional manifolds 
are open and orientable,
and the tangent distributions are also orientable. 
 Since they applied Gromov's \textit{h}-principle for open manifolds, 
then such assumptions were required. 
 In that setting, they obtained the same result 
as Theorem~\ref{thma}, 
and clearly describe the condition 
from the viewpoint of obstruction theory, 
using known topological invariants. 
 Actually, their ``formal'' structure for the Cartan $(2,3,5)$-distribution 
is equivalent to that in this paper, almost Cartan structure 
(see Section~\ref{sec:almCrtn}). 
 Furthermore, they studied necessary and sufficient condition 
for tangent distributions of rank~$2$ to be Cartan $(2,3,5)$-distributions
on open orientable manifolds that satisfy the conditions above. 
 The claims are described by the following topological terms. 
 A manifold $M$ is said to be \emph{spin}\/ if it admits a spin structure. 
 In other words, $M$ satisfies $w_1(M)=0,\ w_2(M)=0$, where $w_i(M)$ denotes 
Stiefel-Whitney class of $TM$. 
 Let $\frac{1}{2}p_1(M)\in H^4(M;\mathbb{Z})$ 
denote the first fractional Pontryagin class of $TM$. 

%
%
\begin{theo}[Dave-Haller]
\textup{(1)}\ Let $M$ be an open $5$-dimensional manifold. 
 The manifold $M$ admits an orientable Cartan $(2,3,5)$-distribution 
if and only if $M$ is spin. 

\noindent
\textup{(2)}\ Let $M$ be an open $5$-dimensional spin manifold 
and $\D$ an orientable tangent distribution of rank~$2$ on $M$. 
 Then there exists an orientable Cartan $(2,3,5)$-distribution $\tilde{\D}$ 
on $M$ with $e(\tilde{\D})=e(\D)$, 
if and only if $e(\D)^2=\frac{1}{2}p_1(M)$, 
where $e(\D)\in H^2(M;\mathbb{Z})$ denotes the Euler class. 
\end{theo} 
\noindent
 In other words, the second statement characterizes the possible Euler classes 
that can be realized by the Cartan $(2,3,5)$-distributions 
on open $5$-dimensional spin manifolds. 

  Even in the case where manifolds are closed, 
Dave and Haller~\cite{davehaller} 
also studied topological requirements 
for the almost generic rank-two distributions in dimension~$5$, 
that they regarded as a formal structure in their sense 
(see Theorem~\ref{thm:dh}). 
 However, Gromov's \textit{h}-principle for open manifolds does not apply 
in this case. 
 Combining their results with those obtained in this paper, 
we obtain the following theorem. 
%
%
\begin{thm}\label{thmb}
\textup{(1)}\ Let $M$ be a closed connected $5$-dimensional manifold. 
 The manifold $M$ admits an orientable Cartan $(2,3,5)$-distribution 
if and only if $M$ is spin 
with vanishing Kervaire semicharacteristic $\kappa(M)=0\in\Z/2\Z$. \\
\textup{(2)}\ Let $M$ be a closed connected $5$-dimensional spin manifold 
with vanishing Kervaire semicharacteristic $\kappa(M)=0$, 
and $\D$ an orientable tangent distribution of rank~$2$ on $M$. 
 Then there exists an orientable Cartan $(2,3,5)$-distribution $\tilde{\D}$ 
on $M$ with $e(\tilde{\D})=e(\D)$, 
if and only if $e(\D)^2=\frac{1}{2}p_1(M)$, 
where $e(\D)\in H^2(M;\mathbb{Z})$ is the Euler class. 
\end{thm} 
\noindent
 In other words, the second statement characterizes the possible Euler classes 
that can be realized by the Cartan $(2,3,5)$-distributions 
on closed connected spin $5$-manifolds 
with vanishing Kervaire semicharacteristic. 

%
%
\begin{exs} 
  From Theorem~\ref{thmb}, we obtain some concrete examples of closed manifolds 
that either admit or do not admit Cartan $(2,3,5)$-distributions. 
  It is well known that $S^3\times S^2$ admits a Cartan $(2,3,5)$-distribution, 
the standard model (see \cite{bcg3}, \cite{montgomery}, and so on). 
 The $5$-dimensional manifold $S^3\times S^2$ is in fact spin, 
and satisfies $\kappa(S^3\times S^2)=0$. 
 On the other hand, the $5$-dimensional sphere $S^5$ 
and the manifold $\mathbb{C}P^2\times S^1$ 
do not admit the Cartan $(2,3,5)$-distribution. 
 $S^5$ is spin but its Kervaire semicharacteristic is nonzero. 
 $\mathbb{C}P^2\times S^1$ 
has vanishing Kervaire semicharacteristic but is not spin. 
 Thus neither of them admits a Cartan $(2,3,5)$-distribution. 
\end{exs}

  Another main goal of this paper is a classification 
of Cartan $(2,3,5)$-distributions. 
 In this paper we apply Gromov's convex integration method 
to show the \textit{h}-principle. 
 As a consequence, 
we obtain the one-parametric version of the \textit{h}-principle. 
 In other words, we obtain the following classification result. 
%
%
\begin{thm}\label{thmc}
  Let $\mathcal{D}_1,\ \mathcal{D}_2\subset TM$ 
be Cartan $(2,3,5)$-distributions on a $5$-dimensional manifold $M$. 
 If they are homotopic through almost Cartan structures, 
then they are homotopic through the genuine Cartan $(2,3,5)$-distributions. 
\end{thm} 

  Concerning the methods to show the existence of geometric structures, 
there are several important works. 
 The existence of a contact structure on a closed orientable $3$-manifold 
is proved by Martinet~\cite{martinet71} by a constructive method. 
 On the other hands, methods from the viewpoints of the \textit{h}-principles 
are applied to some geometric structures. 
 For higher-dimensional manifolds, 
it is proved by Borman, Eliashberg, and Murphy~\cite{boelmu} 
from the viewpoint of the \textit{h}-principle 
that, if a manifold admits an almost contact structure, 
it admits a genuine contact structure. 
 The existence of Engel structure is proved 
by Casals, P\'erez, del Pino, and Presas~\cite{cp3} in a similar way. 
 One of the works that has most strongly influenced this works 
is McDuff's \textit{h}-principle~\cite{mcduff} for even-contact structures. 
 Gromov's convex integration method is applied 
to show the \textit{h}-principle. 

\smallskip

  The present paper is organized as follows. 
 In Section~\ref{sec:Cdistr}, we review the Cartan $(2,3,5)$-distribution 
and introduce the notion of almost Cartan structure. 
 Then we study a characterization of almost Cartan structure. 
 As a result, assuming Theorem~\ref{thma}, we obtain Theorem~\ref{thmb} here. 
 In Section~\ref{sec:h-prin}, we recall the notion of the \textit{h}-principle, 
and review Gromov's convex integration method, 
that is a method to prove the \textit{h}-principles. 
 In Section~\ref{sec:(3,5)-distributions}, 
we study the existence and classification of $(3,5)$-distributions. 
 It plays an important roll in the proof of the main theorems. 
 In Section~\ref{sec:diff-rel}, we reformulate the problems 
in terms of the \emph{h}-principle. 
 Then we determine the differential relation to consider. 
 Finally, in Section~\ref{sec:pf_ethm}, we prove the theorems. 
 We apply Gromov's convex integration method to show that 
the differential relation determined in the previous section satisfies 
the \textit{h}-principles. 

\section{Cartan distribution and almost Cartan structure}
\label{sec:Cdistr}
  In this section, we introduce the notion of almost Cartan structure 
on $5$-dimensional manifolds, 
viewed as a formal structure corresponding 
to the genuine Cartan $(2,3,5)$-distribution. 
 Then we study its properties. 
 First of all, we review the notion of tangent distributions on manifolds 
and the derivation of tangent distributions in Section~\ref{sec:distr}. 
 Next, in Section~\ref{sec:def_235}, we define and review some basic properties 
of the Cartan~$(2,3,5)$-distribution and $(3,5)$-distribution. 
 In Section~\ref{sec:almCrtn}, 
we define the almost Cartan structure as a formal structure 
of the genuine Cartan~$(2,3,5)$-distribution 
from the viewpoint of the \textit{h}-principle. 
 Then we characterize the notion in algebraic form in some specific situation. 
 In Subsection~\ref{sec:algtopCrtn}, 
we introduce some results on the generic rank-$2$ distributions 
proved by Dave and Haller in~\cite{davehaller}. 

\subsection{Tangent distributions and their derivations} \label{sec:distr}
  In this paper, 
we deal with tangent distributions on manifolds as geometric structures. 
 First of all, we introduce the notion of such structures 
and recall some of its basic properties. 
 Readers should consult some references 
(for example, \cite{bcg3}, \cite{montgomery}, \cite{tanaka70}) 
for further study. 
 Let $M$ be a smooth manifold of dimension $n$. 
 A subbundle $\D\subset TM$ of the tangent bundle $TM$ 
with $r$-dimensional fibers 
is called a \emph{tangent distribution} (or \emph{distribution} for short) 
of rank $r$ on $M$. 
 Let $\sect{\D}$ be the set of all cross-sections (vector fields) 
of the subbundle $\D\to M$, and $\DD\subset \ub{TM}$ 
the sheaf of the vector fields. 
 Then $\D$ can be considered as a distribution 
of the $r$-dimensional tangent subspace $\D_p\subset T_pM$ 
at each point $p\in M$, 
or equivalently an $r$-dimensional plane field on $M$. 

  We deal with completely non-holonomic distributions in this paper, 
which are defined as follows. 
 Let $\D$ be a tangent distribution of rank~$r$ 
on an $n$-dimensional manifold $M$. 
 A tangent distribution $\D\subset TM$ is said 
to be \emph{completely non-holonomic} 
if, for any local frame $\{X_1,\dots,X_k\}$ of $\D$, 
its iterated Lie brackets $X_i,[X_i,X_j],[X_i,[X_j,X_l]],\dots$ 
span the tangent bundle $TM$. 
 This condition is observed from the viewpoint of tangent distributions 
derived by the Lie brackets as follows. 
 For a given local vector field $X\in \D$ defined 
on a neighborhood of a point $x\in M$, 
we have the germ $\ub{X}_x\in \DD(x)$ 
as an element of the stalk $\DD(x)$ at $x\in M$. 
 Then by inductively setting 
\begin{align*}
  \DD(x)^2&=([\DD,\DD]+\DD)_x \\
          &:=\operatorname{Span}\{[\ub{X}_x,\ub{Y}_x]+\ub{Z}_x\in\ub{TM}(x)
            \mid \ub{X}_x,\ub{Y}_x,\ub{Z}_x\in\DD(x)\}, \\
  \DD(x)^{l+1}&=([\DD,\DD^l]+\DD^l)_x \\ 
          &:=\operatorname{Span}\{[\ub{X}_x,\ub{Y}_x]+\ub{Z}_x\in\ub{TM}(x)
  \mid \ub{X}_x\in\DD(x),\ub{Y}_x,\ub{Z}_x\in\DD(x)^l\},\quad (l=2,3,\dots), 
\end{align*}
we have a flag 
$\DD\subset\DD^2\subset\dots\subset\DD^l\subset\dots\subset\ub{TM}$ 
of subsheaves. 
 The condition that the tangent distribution $\D\subset TM$ 
is completely non-holonomic 
implies that $\DD^{l_0}=\ub{TM}$ holds for some $l_0\in\mathbb{N}$. 
 We have the subspace $\D^l_x\subset T_{x}M$ as 
$
  \D^l_x:=\{X_x\in T_{x}M\mid \ub{X}_x\in \DD(x)^l\} 
$. 
 If $\dim(\D^l_x)$ is constant for any $x\in M$, 
there corresponds a tangent distribution $\D^l\subset TM$ on $M$. 
 Such distributions are called the \emph{derived distributions}\/ of $\D$. 

  In what follows, we deal with the case 
where each $\D^l$ is a tangent distribution. 
 Then we abuse the notation by writing $\D$, $X$,\dots without under-bars. 

\subsection{The Cartan $(2,3,5)$-distribution and $(3,5)$-distribution}
\label{sec:def_235}
  Now we define the Cartan~$(2,3,5)$-distributions, 
and recall some of its properties. 
%
%
\begin{dfn}
  A \emph{Cartan~$(2,3,5)$-distribution}\/
is a completely non-holonomic distribution $\D$ of rank~$2$ 
on a $5$-dimensional manifold $M$ with the derived flag 
$\D\subset\D^2\subset\D^3=TM$, where $\D^2$ 
is a tangent distribution of rank~$3$. 
 In other words, the derivations $\D$, $\D^2$, and $\D^3$ 
are distributions of rank~$2$, $3$, and $5$ respectively.
\end{dfn}

   We should remark that there are two kinds of derivations 
associated with tangent distributions. 
 Recall that the derived distribution $\D^l$ is defined 
by the Lie bracket of $\D$ itself and $\D^{l-1}$. 
 The system of such derived distributions 
is called the \emph{weak derived system}. 
 On the other hand, there is another sequence 
$\D\subset\D^{(2)}\subset\D^{(3)}\subset\cdots$ of tangent distributions 
derived from the same distribution $\D$ defined by: 
\begin{equation*}
  \D^{(2)}=\D^2:=[\D,\D],\quad 
  \D^{(l+1)}:=[\D^{(l)},\D^{(l)}],\quad l=2,3,4,\dots. 
\end{equation*}
 Such a system is called the \emph{strong derived system}. 
 It is clear that $\D^i\subset\D^{(i)}$, for all $i=3,4,5,\dots$. 
 We remark in particular that $\D^{(3)}=\D^3$. 
 Indeed, if a tangent distribution $\D$ of rank~$2$ 
is locally generated by $X,\ Y$, 
both are generated by $X,Y,[X,Y],[X,[X,Y]],[Y,[X,Y]]$. 
 The important thing here is 
that the Cartan $(2,3,5)$-distribution is defined 
in terms of strong derived system. 

  We next recall  basic properties 
concerning the Cartan~$(2,3,5)$-distributions. 
 Although tangent distributions are considered 
as subbundles of the tangent bundles, 
we deal with them as subbundles of the cotangent bundles. 
 In other words, we observe coframings of tangent distributions 
instead of framings. 
 First, we confirm the setting. 
 Let $\D\subset TM$ be a tangent distribution of rank~$2$ 
on a $5$-dimensional manifold $M$. 
 Let $\mathcal{S}(\D)\subset T^\ast M$ be the bundle of covectors 
that annihilate $\D$. 
 Note that $\mathcal{S}(\D)$ is a vector bundle of fiber dimension~$3$. 
 The distribution $\D$ can be locally described by $1$-forms 
$\alpha_1, \alpha_2, \alpha_3\in\sect{T^\ast M}$ as 
$\D=\{\alpha_1=0,\ \alpha_2=0,\ \alpha_3=0\}$. 

  We now examine the choice of the basis of $\mathcal{S}(\D)$ 
annihilating a tangent distribution $\D$. 
 If $\D$ is the Cartan~$(2,3,5)$-distribution 
then we can take a basis $\{\alpha_1,\ \alpha_2,\ \alpha_3\}$ as follows. 

%
%
\begin{prop}\label{prop:Cartanforms}
  Let $\D\subset TM$ be a tangent distribution of rank~$2$ 
on a $5$-dimensional  manifold $M$. 
 The distribution $\D$ is the Cartan~$(2,3,5)$-distribution 
if and only if $\mathcal{S}(\D)$ has a local basis 
$\{\alpha_1,\ \alpha_2,\ \alpha_3\}$ 
which satisfies the following conditions\textup{:} 
\begin{itemize}
\item $\alpha_1\wedge\alpha_2\wedge\alpha_3\wedge d\alpha_i=0\quad (i=1,2)$, 
\item $\alpha_1\wedge\alpha_2\wedge\alpha_3\wedge d\alpha_3\ne 0$, 
\item $\alpha_1\wedge\alpha_2\wedge d\alpha_1$ 
  and $\alpha_1\wedge\alpha_2\wedge d\alpha_2$ 
  are pointwise linearly independent. 
\end{itemize}
 In this case, $\D^2=[\D,\D]=\{\alpha_1=0,\ \alpha_2=0\}$. 
\end{prop}
\noindent
 Although it is well known 
(see \cite{bryhsu}, \cite{art11} for similar descriptions), 
we review a rough sketch of the proof 
since the meanings of the $5$-forms and $4$-forms above is important 
in the proofs of theorems. 
 As we mentioned above, the Cartan $(2,3,5)$-distribution 
is defined as a strong derived system, 
each derivation step is considered as the Lie square $[\D^j,\D^j]$. 
 Then each step corresponds to the exterior derivatives 
$d\alpha_i$ of coframes because of a formula 
$d\alpha_i(X,Y)=X(\alpha_i(Y))-Y(\alpha_i(X))-\alpha_i([X,Y])$. 
 Therefore, the first two lines of the condition imply 
that $\rk\D^2=3$ and $\D^2=\{\alpha_1=0,\alpha_2=0\}$. 
 In addition, the third line implies $\D^3=TM$. 

  It may be better to explain Proposition~\ref{prop:Cartanforms} 
from the structure equations that are given 
in the original paper~\cite{Car5v} by Cartan. 
%
%
\begin{theo}[Cartan]
  For a Cartan $(2,3,5)$-distribution $\D$ on a $5$-dimensional manifold $M$, 
there exists a local coframing $\{\alpha_1,\alpha_2,\dots,\alpha_5\}$ 
describing $\D$ as $\D=\{\alpha_1=0,\alpha_2=0,\alpha_3=0\}$ 
that satisfies the following conditions: 
\begin{equation}
  \label{eq:str-eqns}
  \begin{aligned}
    d\alpha_1&\equiv \alpha_3\wedge\alpha_4 &&\pmod{\alpha_1,\alpha_2}, \\
    d\alpha_2&\equiv \alpha_3\wedge\alpha_5 &&\pmod{\alpha_1,\alpha_2}, \\
    d\alpha_3&\equiv \alpha_4\wedge\alpha_5 &&\pmod{\alpha_1,\alpha_2,\alpha_3}.
  \end{aligned}
\end{equation}
 In this case $\D^2=[\D,\D]=\{\alpha_1=0,\alpha_2=0\}$. 
\end{theo}

  Another important type of tangent distributions in this paper 
is $(3,5)$-distribution. 
 A tangent distribution $\mathcal{E}$ of rank~$3$ 
on a $5$-dimensional manifold $M$ is said to be of \emph{type~$(3,5)$}, 
or is called a \emph{$(3,5)$-distribution}, 
if it satisfies $[\mathcal{E},\mathcal{E}]_x=T_xM$ at any point $x\in M$. 
 In other words, $\rk \mathcal{E}^2_x=5$. 
 Such distributions are closely related to the Cartan~$(2,3,5)$-distributions. 
 In fact, for the Cartan~$(2,3,5)$-distribution $\D$ 
on a $5$-dimensional manifold,  
the derived distribution $\mathcal{E}:=\D^2$ 
is a $(3,5)$-distribution. 
 In order to discuss the \textit{h}-principle 
for the Cartan~$(2,3,5)$-distributions, 
we need those for $(3,5)$-distributions. 
 We discuss them in Section~\ref{sec:(3,5)-distributions}. 
 Some generalization of such distributions are studied 
in~\cite{art11}, \cite{art24}. 

 The $(3,5)$-distributions are also described locally by pair of $1$-forms 
like the Cartan~$(2,3,5)$-distribution. 
 A general case of the following proposition 
is discussed in~\cite{art11}). 
%
%
\begin{prop}\label{prop:dbasis}
  Let $\mathcal{E}\subset TM$ be a distribution of rank~$3$ 
on a $5$-dimensional  manifold $M$. 
 The distribution $\mathcal{E}$ is of type~$(3,5)$
if and only if $\mathcal{S}(\mathcal{E})$ has a local basis 
$\{\alpha_1,\ \alpha_2\}$ 
which satisfies that the following two $4$-forms 
are pointwise linearly independent\textup{:} 
\begin{equation*}
  \alpha_1\wedge\alpha_2\wedge d\alpha_1, \qquad 
  \alpha_1\wedge\alpha_2\wedge d\alpha_2. 
\end{equation*}
\end{prop} 

  In this paper we focus on tangent distributions of these types. 
 Similar types of distributions are intensively studied 
in~\cite{mar-agu_delPino} and \cite{mar-agu} in a similar way. 

\subsection{Almost Cartan structure}\label{sec:almCrtn}
  We introduce the notion of almost Cartan structure on $5$-manifolds, 
in this subsection. 
 Recall that from the Cartan~$(2,3,5)$-distribution $\D$ 
on a $5$-dimensional manifold, 
there is an induced sequence $\D\subset\D^2\subset\D^3=TM$ 
of tangent distributions (see Definition in Section~\ref{sec:distr}). 
 Since the sequence is derived by the Lie brackets, 
there corresponds differential forms with certain properties 
(see Proposition~\ref{prop:Cartanforms}). 
 By forgetting the exterior derivatives, we define the notion of 
almost Cartan structure 
as a formal structure for the Cartan~$(2,3,5)$-distribution as follows. 
%
%
\begin{df}[almost Cartan structure]\label{def:alm_cartan}
  An \emph{almost Cartan structure\/} on a $5$-dimensional manifold $M$ 
is defined as a tuple $(\D\subset\mathcal{E},\{\omega_1,\omega_2;\omega_3\})$, 
where $\D\subset\mathcal{E}$ is a sequence of tangent distributions 
of ranks~$2$ and~$3$ respectively on $M$, 
and $\omega_1$, $\omega_2$, $\omega_3$ are $2$-forms 
that satisfy the following conditions: 
\begin{enumerate}\setlength{\itemindent}{1.7em}
\item[(1)-(i)\phantom{i}] $\omega_1$ and $\omega_2$ are degenerate on $\D$, 
\item[(1)-(ii)] $\omega_3$ is non-degenerate on $\D$, 
\item[(2)\phantom{-(ii)}] $\omega_1$ and $\omega_2$ 
are pointwise linearly independent on $\mathcal{E}$. 
\end{enumerate}
\end{df}
\noindent
 Note that in a precise sense, the $2$-form $\omega_3$ takes values in $TM/\D$, 
and those $\omega_1,\omega_2$ take values in $TM/\mathcal{E}$ 
due to orientations. 
 We should remark that the ``genuine'' Cartan $(2,3,5)$-distribution 
is an almost Cartan structure, in some sense. 
 Indeed, if $\D\subset TM$ is a genuine Cartan $(2,3,5)$-distribution 
on a $5$-dimensional manifold $M$, 
from Proposition~\ref{prop:Cartanforms}, 
there locally exists a triple $(\alpha_1,\ \alpha_2,\ \alpha_3)$ of $1$-forms 
that defines $\D$ as $\D=\{\alpha_1=0,\ \alpha_2=0,\ \alpha_3=0\}$ satisfying 
that the following conditions: 
\begin{itemize}
\item $\alpha_1\wedge\alpha_2\wedge\alpha_3\wedge d\alpha_i=0\quad (i=1,2)$, 
\item $\alpha_1\wedge\alpha_2\wedge\alpha_3\wedge d\alpha_3\ne 0$, 
\item $\alpha_1\wedge\alpha_2\wedge d\alpha_1$ 
  and $\alpha_1\wedge\alpha_2\wedge d\alpha_2$ 
  are pointwise linearly independent. 
\end{itemize}
 This implies the $2$-forms $d\alpha_1$, $d\alpha_2$ and $d\alpha_3$ determine 
three $(TM/\D)$-valued $2$-forms $[d\alpha_1], [d\alpha_2]$, and $[d\alpha_3]$ 
that make $\D$ into an almost Cartan structure 
$(\D\subset\D^2,\{[d\alpha_1],[d\alpha_2];[d\alpha_3]\})$. 

%
%
%
%
%
  If the tangent distribution $\D$ is orientable, 
almost Cartan structure is characterized as follows. 
 We remark that if a manifold admits an orientable almost Cartan structure, 
the manifold should be orientable 
from the definition of almost Cartan structure. 
%
%
\begin{prop}\label{prop:ori-almctn}
  Let $M$ be an orientable $5$-dimensional manifold. 
 $M$ admits an orientable almost Cartan structure 
$(\D\subset\mathcal{E},\{\omega_1,\omega_2;\omega_3\})$ 
if and only if  
the tangent bundle of $M$ splits as $TM\cong\xi\oplus\epsilon^1\oplus\xi$, 
where $\xi$ is an orientable tangent distribution of rank~$2$, 
and $\epsilon^1$ is a trivial line bundle. 
\end{prop}

\begin{proof}
  First, assume that $M$ admits the splitting 
$TM\cong\xi\oplus\epsilon^1\oplus\xi$. 
 Since $\xi$ is orientable, 
we may locally choose oriented frames $\{u_1, u_2\}$, $\{v_1, v_2\}$ 
for the two copies of $\xi$. 
 Similarly the trivial line bundle $\epsilon^1$ has its framing $\{\ell\}$. 
 In other words, we have 
$TM\cong\spn{u_1,u_2}\oplus\spn{\ell}\oplus\spn{v_1,v_2}$. 
 Let $\beta_1, \beta_2, \gamma, \alpha_1,\alpha_2$ be the dual coframe
to $u_1, u_2, \ell, v_1,v_2$. 
 Then set 
\begin{align*}
  &\D:=\xi=\spn{u_1,u_2}=\{\alpha_1=0, \alpha_2=0, \gamma=0\}, \\
  & \mathcal{E}:=\xi\oplus\epsilon^1=\spn{u_1,u_2}\oplus\spn\ell
    =\{\alpha_1=0, \alpha_2=0\}, 
\end{align*}
and define 
\begin{align*}
  &\omega_1:=\gamma\wedge\beta_1,\quad \omega_2:=\gamma\wedge\beta_2,\quad 
    \omega_3:=\beta_1\wedge\beta_2. 
\end{align*}
 Then $(\D\subset\mathcal{E},\{\omega_1,\omega_2;\omega_3\})$ 
is an orientable almost Cartan structure on $M$. 
 Indeed 
\begin{align*}
  &\omega_i\vert_\D=(\gamma\wedge\beta_i)\vert_{\spn{u_1,u_2}}=0, (i=1,2), \\ 
  &\omega_3\vert_\D=(\beta_1\wedge\beta_2)\vert_{\spn{u_1,u_2}}
    =\beta_1\wedge\beta_2 \quad \text{is non-degenerate} \\
  &\omega_i\vert_\mathcal{E}=(\gamma\wedge\beta_i)\vert_{\spn{u_1,u_2,l}}
    =\gamma\wedge\beta_i,\ (i=1,2) \quad 
    \text{are pointwise linearly independent}. 
\end{align*}
 It is instructive to compare this construction 
with the structure equations, Equations~\eqref{eq:str-eqns}, 
of the Cartan $(2,3,5)$-distributions. 

  Conversely, assume that there exists an orientable almost Cartan structure 
$(\D\subset\mathcal{E}, \{\omega_1,\omega_2;\omega_3\})$ 
on an orientable $5$-dimensional manifold $M$. 
 Then the tangent bundle splits as 
\begin{equation}\label{eq:presplitting}
  TM\cong\D\oplus(\mathcal{E}/\D)\oplus TM/\mathcal{E}. 
\end{equation}
 Then we consider each bundle in turn. 

\paragraph{(1) The bundle $\D$. }\ 
 The first bundle $\D$ is an orientable tangent distribution of rank~$2$. 
 We take this $\D\subset TM$ itself as $\xi$. 

\paragraph{(2) The bundle $\mathcal{E}/\D$. }\ 
 Since $\D$ is an orientable tangent distribution of rank~$2$, 
and $\omega_3\vert_\D$ is non-degenerate, 
$\epsilon^1:=\ker\omega_3$ is a trivial line bundle 
and $\mathcal{E}=\D\oplus\epsilon^1$. 
 Therefore we have $\mathcal{E}/\D\cong\epsilon^1$. 

\paragraph{(3) The bundle $TM/\mathcal{E}$. }\ 
 We show that $TM/\mathcal{E}\cong\D$. 
 Since there exists a $2$-form $\omega_3$ 
which is non-degenerate on $\D$, 
the bundle $\D$ is isomorphic to its dual bundle $\D^\ast$, 
that is, $\D\cong \D^\ast$. 
 Therefore it suffices to show that $TM/\mathcal{E}\cong \D^\ast$. 

  In order to do that, we first choose a local coframe for $\D^\ast$. 
 From the definition of almost Cartan structure 
(see Definition~\ref{def:alm_cartan}), 
the $2$-forms $\omega_1$, $\omega_2$ are degenerate 
along the distribution $\D$ of rank~$2$ 
and non-degenerate along the distribution $\mathcal{E}\cong\D\oplus\epsilon^1$ 
of rank~$3$. 
 In addition, they are pointwise linearly independent on $\mathcal{E}$. 
 Then, by linear algebra in dimension~$3$, 
we can choose a suitable local coframe 
$\{\eta, \theta_1, \theta_2\}$ on $\mathcal{E}$ 
so that 
\begin{equation*}
  \omega_i=\eta\wedge\theta_i,\quad i=1,2, 
\end{equation*}
where $\eta, \theta_i$ are $1$-forms that satisfy 
\begin{equation*}
  \eta\vert_\D=0,\quad \eta\vert_{\epsilon^1}\ne 0,\quad 
  \theta_i\vert_{\epsilon^1}=0. 
\end{equation*}
 Since $\omega_1,\ \omega_2$ are linearly independent, 
$\theta_1, \theta_2$ are linearly independent on $\D$. 
 Then, in this sense, $\{\theta_1, \theta_2\}$ is considered 
as a basis of $\D^\ast$. 

  Next, we construct an isomorphism between $TM/\mathcal{E}$ and $\D^\ast$. 
 Because $M$ and $\mathcal{E}$ are orientable, 
the quotient bundle $TM/\mathcal{E}$ is also an orientable bundle of rank~$2$. 
 Then we take a local oriented frame $\{v_1,\ v_2\}$ of $TM/\mathcal{E}$ 
and use the same notation for their representative in $TM$. 
 Then, using these bases, we can define a pointwise linear mapping 
\begin{equation*}
  \phi_p\colon T_pM/\mathcal{E}_p\to\D^\ast_p, \qquad
  \phi_p(v_i)=\theta_i\vert_{\D_p},\ i=1,2.
\end{equation*}
 By construction, $\phi_p$ is a linear isomorphism at each $p\in M$, 
and hence induces the bundle isomorphism $\phi\colon TM/\mathcal{E}\to\D^\ast$.

 Now we have examined all the bundles involved in the splitting 
in Equation~\ref{eq:presplitting}. 
 By setting $\D=\xi$, 
we obtain that $\mathcal{E}/\D=:\epsilon^1$ is a trivial line bundle, 
and that $TM/\mathcal{E}\cong\xi$. 
 Consequently, we conclude that we have the splitting 
\begin{equation*}
  TM\cong \xi\oplus\epsilon^1\oplus\xi, 
\end{equation*}
for an orientable distribution $\xi$ and a trivial line bundle $\epsilon^1$. 
%
%
%
%
%
%
\end{proof}

\subsection{Algebraic topology on almost Cartan structures}
\label{sec:algtopCrtn}
  When the tangent distributions are orientable, 
and hence the $5$-dimensional manifolds are orientable, 
 Dave and Haller~\cite{davehaller} studied the obstruction-theoretic condition
for the existence of the splitting $TM\cong \xi\oplus\epsilon^1\oplus\xi$. 
 In this subsection, we introduce such a result that is 
closely related to results in this paper. 

  In order to describe the theorem, we introduce some obstructions. 
 A manifold $M$ is said to be \emph{spin}\/ if it admits a spin structure. 
 In other words, this means that $w_1(M)=0$ and $w_2(M)=0$, 
where $w_i(M)\in H^i(M;\mathbb{Z}/2\mathbb{Z})$ 
is the $i$-th Stiefel-Whitney class of $M$. 
 For a spin manifold $M$, let $\frac{1}{2}p_1(M)\in H^4(M;\mathbb{Z})$ 
be the first fractional Pontryagin class of $M$. 
 For a closed connected manifold $N$, let $\kappa(N)\in\mathbb{Z}/2\mathbb{Z}$ 
be the (real) Kervaire semicharacteristic of $N$, 
that is defined (see \cite{kervaire}) as
\begin{equation*}
  \kappa(N)=\sum_{q:\text{even}}\dim H^q(N;\mathbb{R}) \quad\bmod 2. 
\end{equation*}
 For a closed connected spin manifold $M$, 
it is known (see \cite{lusmilpet}) 
that $\kappa(M)$ coincides with the $\mathbb{Z}_2$ semicharacteristic $k_2(M)$, 
\begin{equation*}
  \kappa(M)=\kappa_2(M):=\sum_{q:\text{even}}\dim H_q(M;\mathbb{Z}/2\mathbb{Z}). 
\end{equation*}

  Then the necessary and sufficient condition 
for the existence of the splitting $TM\cong \xi\oplus\epsilon^1\oplus\xi$, 
due to Dave and Haller~\cite{davehaller}, is described as follows. 
%
%
\begin{thrm}[\cite{davehaller}]\label{thm:dh}
  \textup{(I)} \textup{(Open case)}\ 
 Let $M$ be an open connected $5$-dimensional manifold. \\
\textup{(1)}\  The tangent bundle of $M$ splits as 
$TM\cong\xi\oplus\epsilon^1\oplus\xi$ for some orientable tangent distribution 
or subbundle $\xi$ of rank~$2$ on $M$, 
if and only if $M$ is spin. \\
\textup{(2)}\ Suppose that $M$ is spin. 
 Let $\xi$ be an orientable tangent distribution of rank~$2$ on $M$. 
 Then $M$ admits the splitting $TM\cong\xi\oplus\epsilon^1\oplus\xi$ 
if and only if $\frac{1}{2}p_1(M)=e(\xi)^2$, 
where $e(\xi)\in H^2(M;\mathbb{Z})$ is the Euler class. 

\noindent
\textup{(II)} \textup{(Closed case)}\ 
 Let $M$ be a closed connected $5$-dimensional manifold. \\ 
\textup{(1)}\  The tangent bundle of $M$ splits as 
$TM\cong\xi\oplus\epsilon^1\oplus\xi$ for some orientable tangent distribution 
or subbundle $\xi$ of rank~$2$ on $M$, 
if and only if 
$M$ is spin with vanishing Kervaire semicharacteristic $\kappa(M)=0$. \\
\textup{(2)}\ Suppose that $M$ is spin 
with the vanishing Kervaire semicharacteristic $\kappa(M)=0$. 
 Let $\xi$ be an orientable tangent distribution of rank~$2$ on $M$. 
 Then $M$ admits the splitting $TM\cong\xi\oplus\epsilon^1\oplus\xi$ 
if and only if $\frac{1}{2}p_1(M)=e(\xi)^2$. 
\end{thrm}

  In the case where $5$-dimensional manifolds 
and tangent distributions of rank~$2$ are orientable, 
now we have a characterization of almost Cartan structures 
(Proposition~\ref{prop:ori-almctn}) 
and an obstruction-theoretic description 
of the existence of an almost Cartan structure (Theorem~\ref{thm:dh}). 
 Then we show Theorem~\ref{thmb} assuming Theorem~\ref{thma}. 

\begin{proof}[Proof of Theorem~\ref{thmb}]
 Let $M$ be an orientable closed connected $5$-dimensional manifold. 

\noindent
(1)\ Theorem~\ref{thma} implies 
that the existence of an orientable genuine Cartan $(2,3,5)$-distribution 
is equivalent to that of an orientable almost Cartan structure. 
 The existence of an orientable almost Cartan structure 
is equivalent to that of the splitting $TM\cong\xi\oplus\epsilon^1\oplus\xi$, 
on account of Proposition~\ref{prop:ori-almctn}. 
 Then, from Theorem~\ref{thm:dh}, 
we conclude that $M$ admits an orientable genuine Cartan $(2,3,5)$-distribution 
if and only if $M$ is spin with vanishing Kervaire semicharacteristic 
$\kappa(M)=0\in\mathbb{Z}/2\mathbb{Z}$. 

\noindent
(2)\ Suppose that $M$ is spin with vanishing Kervaire semicharacteristic 
$\kappa(M)=0$. 
 From part~(1) of Theorem~\ref{thmb}, that we have just proved, 
there exists an orientable Cartan $(2,3,5)$-distribution. 

  First, let $\D$ be an orientable Cartan $(2,3,5)$-distribution. 
 Then, there induced an orientable almost Cartan structure 
$(\D\subset\D^2,\{\omega_1,\omega_2;\omega_3\})$ 
by the discussion after the definition of the almost Cartan structure 
(or Theorem~\ref{thma}). 
 This implies that the tangent bundle $TM$ split as 
$TM\cong\D\oplus \epsilon^1\oplus\D$ by Proposition~\ref{prop:ori-almctn}. 
 Then, on account of Theorem~\ref{thm:dh}, 
we obtain $e(\D)^2=\frac{1}{2}p_1(M)$. 

  Next, assume that $e(\D)^2=\frac{1}{2}p_1(M)$ 
for an orientable tangent distribution $\D$ 
that is not exactly a Cartan $(2,3,5)$-distribution. 
 By Theorem~\ref{thm:dh}, the tangent bundle $TM$ split 
as $\D\oplus\epsilon^1\oplus\D$. 
 Proposition~\ref{prop:ori-almctn} implies 
that there exists an open almost Cartan structure 
$(\D\subset\mathcal{E},\{\omega_1,\omega_2;\omega_3\})$. 
 Then, on account of Theorem~\ref{thma}, 
there exists an orientable genuine Cartan $(2,3,5)$-distribution $\tilde{\D}$ 
that is homotopic to the given $\D$. 
 Then we obtain $e(\tilde{\D})=e(\D)$. 
\end{proof}
\noindent
 Then it remains to show Theorem~\ref{thma} and Theorem~\ref{thmc}. 

\section{Homotopy principle and Gromov's convex integration method}
\label{sec:h-prin}
  We introduce the notion of the \textit{h}-principles 
in Section~\ref{sec:diff_rel}. 
 It is a key tool to show Theorems~\ref{thma} and~\ref{thmc}. 
 In order to show that our objects corresponding to geometric structures 
satisfy the \textit{h}-principles, 
we use the convex integration method due to Gromov. 
 In Section~\ref{sec:cvx_itgr}, we introduce that method. 
 In order to make this paper self-contained, 
some basic things that are needed in this paper are defined there. 
 Readers should refer to the literature 
for further study on the \textit{h}-principles and convex integration method 
(see~\cite{gromov_pdr}, \cite{elmi_book}, \cite{madachi}, \cite{spring}). 

\subsection{Differential relations and homotopy principle}\label{sec:diff_rel}
 We introduce the notion of the \emph{h}-principle
for partial differential relations.
 It is a general viewpoint for solving certain problems 
related to differentiation.

  First, we review basic notions concerning fibrations and jets. 
 Let $p\colon X\to M$ be a fibration over a manifold $M$. 
 The \emph{$r$-jet extension} $p^r\colon X^{(r)}\to M$ 
of the fibration $p\colon X\to M$ 
is defined as the manifold $X^{(r)}$ consisting of all $r$-jets 
of cross-sections $M\to X$ of the fibration at all $v\in V$ 
and the projection $p^r\colon X^{(r)}\to M$. 
 $X^{(r)}$ is called the \emph{$r$-jet space}. 
 The cross-section $J^r_f\colon M\to X^{(r)}$ of $p^r\colon X^{(r)}\to M$ 
is the $r$-jet extension of a cross-section $f\colon M\to X$. 

  Then, we introduce a description method for certain problems. 
 Let $p\colon X\to M$ be a fibration. 
 A \emph{differential relation} of order $r$ for sections 
of $p\colon X\to M$ 
is defined as a subset $\mathcal{R}$ of the $r$-jet space $\rjet X$. 
 For example, a system of differential equations can be regarded 
as a subset of the jet space $J^r(\mathbb{R}^n,\mathbb{R}^m)$, 
that is, a differential relation. 
 There are two kinds of solutions to a differential relation. 
 A \emph{formal solution}\/ 
to the differential relation $\mathcal{R}\subset\rjet{X}$ 
is defined to be a section $F\colon M\to\rjet{X}$ 
of the fibration $p^r\colon \rjet{X}\to M$ 
that satisfies $F(M)\subset\mathcal{R}$. 
 Let $\secc\mathcal{R}$ denote the space of formal solution to $\mathcal{R}$. 
 On the other hand, a \emph{genuine solution}\/ to $\mathcal{R}\subset\rjet{X}$ 
is defined to be a section $f\colon M\to X$ of $p\colon X\to M$ 
whose $r$-jet satisfies $(J^r_f)(M)\subset\mathcal{R}$. 
 Let $\operatorname{Sol}\mathcal{R}$ denote the space of genuine solutions 
to $\mathcal{R}$. 

  Some examples of differential relations come from singularity theory. 
 Roughly speaking, in some cases, 
a singular point 
is a point 
where certain expressions related to derivatives vanish. 
 Then, to singular points, there corresponds the set $\Sigma\subset\rjet X$ 
called the \emph{singularity}. 
 Setting $\mathcal{R}:=\rjet X\setminus\Sigma$, 
we have an open differential relation. 
 In many cases, it is important for singularity theory 
to solve such differential relations. 

  Now, we introduce the notion of the homotopy principle. 
 In general, the formal solvability of a differential relation 
is just a necessary condition 
for the genuine solvability of the differential relation. 
 In some cases, the formal one is sufficient for the genuine solvability. 
 The notion of the homotopy principle was introduced by Gromov and Eliashberg 
to formalize such phenomena (see~\cite{groeli71}). 
 Let $p\colon X\to M$ be a fibration, 
and $\mathcal{R}\subset \rjet X$ a differential relation. 
 $\mathcal{R}$ is said to satisfy the \emph{homotopy principle}\/ 
(or \emph{h-principle}\/ for short) 
if every formal solution to $\mathcal{R}$ is homotopic in $\secc{\mathcal{R}}$ 
to a genuine solution to $\mathcal{R}$. 
 In this paper, we use the notion to show the existence of a genuine solution. 
 In addition to that, 
there are some flavors of the \textit{h}-principle (see~\cite{elmi_book}). 
 In this paper we introduce one of them. 
 The differential relation $\mathcal{R}\subset\rjet{X}$ is said to satisfy 
the \emph{one-parametric \textit{h}-principle}\/ 
if every family $\{f_t\}_{t\in[0,1]}$ of formal solutions to $\mathcal{R}$ 
between genuine solutions $f_0,\ f_1$ 
can be deformed inside $\secc{\mathcal{R}}$ 
to a family $\{\tilde{f}_t\}_{t\in[0,1]}$ of genuine solutions to $\mathcal{R}$ 
keeping both the ends $f_0,\ f_1$. 
 The notion is used to classifications. 

  In general, it is difficult to show 
if a differential relation satisfies the \emph{h}-principles or not. 
 For each relation, we should apply each suitable method. 
 We introduce one of such methods in the next section.

\subsection{Convex integration method}\label{sec:cvx_itgr}
  In this section, we introduce a method due to Gromov 
to show the \emph{h}-principles 
for certain differential relations, 
which is called the \emph{convex integration} method. 

  First, we recall ampleness of subsets. 
 Let $P$ be an affine space, and $\Omega\subset P$ a subset. 
 The minimum convex set that includes $\Omega$ 
is called the \emph{convex hull}\/ of $\Omega$. 
 The subset $\Omega\subset P$ is said to be \emph{ample}\/ 
if the convex hull of each path-connected component of $\Omega$ is equal to $P$.
 The empty set is also defined to be ample. 

  In order to define the ampleness of differential relations, 
we first introduce the principal directions for fibrations. 
 Let $p\colon X\to M$ be a fibration over an $n$-dimensional manifold $M$ 
with the fiber dimension~$q$. 
 For the $1$-jet $p^1\colon \ojet X\to M$, 
there exists the natural projection $p^1_0\colon\ojet{X}\to X^{(0)}=X$, 
which is an affine bundle. 
 Let 
\begin{equation*}
  E_x:=(p^1_0)^{-1}(x),\quad x\in X
\end{equation*}
denote the fiber of $p^1_0\colon \ojet X\to  X$ over $x\in X$, 
and $\vertt_x\subset T_xX$ the $q$-dimensional tangent space at $x\in X$ 
of the fiber $p^{-1}(p(x))\subset X$ of $p\colon X\to M$ over $p(x)\in V$. 
 Then the fiber $E_x\subset X$ can be identified with 
\begin{equation*}
  \operatorname{Hom}(T_{p(x)}V,\vertt_x)
  \cong\operatorname{Hom}(\mathbb{R}^n,\mathbb{R}^q)
  \cong\{q\times n\ \text{matrices}\} 
\end{equation*}
as follows. 
 As $\ojet{X}$ is the $1$-jet space, a point of $\ojet{X}$ is regarded 
as a non-vertical $n$-dimensional subspace $P_x\subset T_xX$, 
where ``non-vertical'' implies that $P_x$ is transverse 
to $\vertt_x\subset T_xX$. 
 Then, regarding $P_x\subset T_xX\cong T_{p(x)}V\oplus\vertt_x$ as a graph, 
there corresponds a linear mapping. 
 Then, for a fixed point $x\in X$, 
the fiber $E_x\subset\ojet{X}$ is identified with 
$\operatorname{Hom}(T_{p(x)}V, \vertt_x)$. 
 Now we define the principal directions for fibrations. 
 Suppose that a hyperplane $\tau\subset T_{p(x)}V$ 
and a linear mapping $l\colon \tau\to \vertt_x$ are given. 
 For these $\tau$, $l$, 
let $P^l_\tau\subset E_x$ denote the affine subspace of $E_x$ defined as 
\begin{equation*}
  P^l_\tau:=\{L\in\operatorname{Hom}(T_{p(x)}V, \vertt_x)\mid L|_\tau=l\}
  \subset E_x. 
\end{equation*}
 Such affine subspaces of $E_x$ are said to be \emph{principal}. 
 Note that the principal subspaces are parallel $q$-dimensional affine spaces 
if the hyperplane $\tau\subset T_{p(x)}V$ is fixed. 
 The direction of the principal subspaces $P^l_\tau$ determined by $\tau$ 
is called the \emph{principal direction}. 

  Now we define the ampleness of differential relations. 
 Let $p\colon X\to M$ be a fibration 
and $\mathcal{R}\subset\ojet{X}$ a differential relation. 
 A differential relation $\mathcal{R}$ is said to be \emph{ample}\/ 
if the intersection of $\mathcal{R}$ with any principal subspaces is ample 
in the affine fiber. 

  Then the key tool in this paper is the following theorem due to Gromov 
(see~\cite{gromov73}, \cite{gromov_pdr}, \cite{elmi_book}, \cite{madachi}, 
\cite{spring}). 

%
%
\begin{thrm}\label{thm:HP_apl_dr}
  Let $p\colon X\to M$ be a fibration and 
$\mathcal{R}\subset\ojet X$ an open differential relation. 
 If $\mathcal{R}\subset \ojet X$ is ample 
then $\mathcal{R}$ satisfies the \textit{h}-principle. 
 In addition, under this condition, $\mathcal{R}$ also satisfies 
the one-parametric \textit{h}-principle. 
\end{thrm} 

%
%
\begin{rmk}
  Gromov showed by this method 
that other flavors of the \textit{h}-principles also holds 
from the ampleness of open differential relations. 
(see~\cite{gromov73}, \cite{elmi_book})
\end{rmk}

  As we see in Section~\ref{sec:diff_rel}, 
singularity theory is a rich source of open differential relations. 
 Let $p\colon X\to M$ a fibration. 
 A singularity $\Sigma\subset\ojet{X}$ is said to be \emph{thin}\/ 
if, at any point $a\in\Sigma$, the intersection $P\cap \Sigma$ 
with any principal subspace $P$ through $a\in\Sigma$ 
is a stratified subset of dimension greater than or equal to $2$ in $P$. 
 When the singularity $\Sigma\subset\ojet{X}$ is thin, 
the open differential relation $\mathcal{R}:=\ojet{X}\setminus\Sigma$ is ample. 
 Then, as a corollary of Theorem~\ref{thm:HP_apl_dr}, we obtain the following. 
%
%
\begin{cor}
  Let $p\colon X\to M$ be a fibration. 
 If a singularity $\Sigma\subset\ojet{X}$ is thin, 
then the differential relation $\mathcal{R}:=\ojet{X}\setminus\Sigma$ 
satisfies the \textit{h}-principle. 
 In addition, $\mathcal{R}$ also satisfies 
the one-parametric \textit{h}-principle under the same condition. 
\end{cor} 

  Next, we introduce the homotopy principle for differential sections 
for a linear differential operator $\mathcal{F}$, following~\cite{elmi_book}. 
 Let $p_x\colon X\to M$ and $p_z\colon Z\to M$ be vector bundles 
over a manifold $M$. 
 We first introduce an operator 
$\mathcal{F}\colon \sect{X}\to \sect Z$ 
called a first order linear differential operator as follows. 
 Let $F\colon\ojet{X}\to Z$ be a fiberwise homomorphism 
between vector bundles over $M$. 
 Then there corresponds a mapping $\tilde{F}\colon \sect{\ojet{X}}\to \sect Z$ 
as $\sigma\mapsto F\circ \sigma$ 
for a section $\sigma$ of the $1$-jet extension $(p_x)^1\colon\ojet{X}\to M$. 
 By the composition with the differential operator 
$J^1\colon\sect X\to\sect{\ojet{X}}$, 
we have the linear operator 
\begin{equation*}
  \mathcal{F}:=\tilde{F}\circ J^1\colon\sect X\to\sect Z. 
\end{equation*}
 The operators of this type are called 
the \emph{first order linear differential operators}. 
 The bundle homomorphism $F$ 
for a first order linear differential operator $\mathcal{F}$ 
is called the \emph{symbol}\/ of $\mathcal{F}$. 
 Let $\symb\mathcal{F}=F$ denote it. 

  For example, the exterior derivative of differential $p$-forms 
\begin{equation*}
  d\colon\sect{\bigwedge^pT^\ast M}\to \sect{\bigwedge^{p+1}T^\ast M}
\end{equation*}
is a first order linear differential operator. 
 Its symbol $D:=\symb d$ is a fiberwise epimorphism 
$D\colon \ojet{(\bigwedge^pT^\ast M)}\to \bigwedge^{p+1}T^\ast M$. 

  We now introduce the homotopy principle for differential sections. 
 Let $p_x\colon X\to M$ and $p_z\colon Z\to M$ be vector bundles 
over a manifold $M$, 
and $\mathcal{F}\colon \sect X\to \sect Z$ 
a first order linear differential operator. 
 An \emph{$\mathcal{F}$-section}\/ is, by definition, 
a section $s_z\colon M\to Z$ of the bundle $p_z\colon Z\to M$ 
for which there exists a section $s_x\colon M\to X$ of $p_x\colon X\to M$ 
that satisfies $s_z=\mathcal{F}(s_x)$. 
 For example, when $\mathcal{F}=d$ is the exterior derivative 
$d\colon \sect{\bigwedge^{p-1}T^\ast M}\to \sect{\bigwedge^pT^\ast M}$ 
of differential $(p-1)$-forms, 
the $\mathcal{F}$-sections are exact differential $p$-forms. 
 In order to state a property of differential sections, 
we introduce the following notation. 
 For the given subset $S\subset Z$, let $\secc_\mathcal{F}(Z\setminus S)$ denote 
the space of $\mathcal{F}$-sections $s\colon M\to Z$ of $p_z\colon Z\to M$ 
that satisfies $s(V)\subset Z\setminus S$. 
 Then the following theorem concerning the \textit{h}-principle 
for differential sections is introduced in~\cite{elmi_book}. 
%
%
\begin{thrm}\label{thm:HP_D-secs}
  Let $\mathcal{F}\colon \sect X\to \sect Z$ 
be a first order linear differential operator 
for vector bundles $p_x\colon X\to M$ and $p_z\colon Z\to M$ 
over a manifold $M$. 
 Assume that the symbol $F=\symb\mathcal{F}\colon\ojet{X}\to Z$ 
is fiberwise epimorphic. 
 For the given subset $S\subset Z$, set 
\begin{equation*}
  \Sigma:=F^{-1}(S)\subset\ojet{X}.
\end{equation*}
 If the differential relation $\mathcal{R}:=\ojet{X}\setminus\Sigma$ 
satisfies the \textit{h}-principle, 
then the inclusion mapping 
$\dsec{Z\setminus S}\hookrightarrow\sect{Z\setminus S}$ 
also satisfies the \textit{h}-principle. 
 In other words, any section $s_0\in\sect{Z\setminus S}$ is homotopic 
in $\sect{Z\setminus S}$ to a section $s_1\in\dsec{Z\setminus S}$. 
 In addition, the same claim also holds 
for the one-parametric \textit{h}-principle. 
\end{thrm}

\section{The \textit{h}-principles for distributions of type~$(3,5)$}
\label{sec:(3,5)-distributions}
  In this section, we discuss existence and classification 
of tangent distributions of type~$(3,5)$. 
 Although more general cases of such problems are discussed in~\cite{art24}, 
we restrict ourselves to the case of type~$(3,5)$, 
to keep this paper self-contained. 
 As we have mentioned in Section~\ref{sec:intro}, 
if $\D$ is the Cartan~$(2,3,5)$-distribution, 
then the derived distribution $\mathcal{E}=[\D,\D]$ is a $(3,5)$-distribution. 
 Then the discussion in this section is important for the proofs 
of the main theorems in Section~\ref{sec:pf_ethm}. 

  We apply Gromov's convex integration method introduced 
in the previous section.
 In Subsection~\ref{sec:claim(3,5)}, 
we introduce the existence and classification claims, 
Theorem~\ref{thm:thm_(3,5)}, for $(3,5)$-distributions. 
 The statement is reformulated into the key theorem 
in terms of the \textit{h}-principles. 
 Then we discuss the differential relation relevant to the key theorem 
in Subsection~\ref{sec:key_(3,5)}. 
 Applying Gromov's convex integration method, 
we prove Theorem~\ref{thm:key_(3,5)}, and then Theorem~\ref{thm:thm_(3,5)} 
in Subsection~\ref{sec:proof_(3,5)}.

\subsection{Claims for $(3,5)$-distributions}\label{sec:claim(3,5)}
  First, we define a formal structure for $(3,5)$-distribution. 
 Like almost Cartan structure to the Cartan~$(2,3,5)$-distribution, 
an almost $(3,5)$-distribution is defined as follows. 
%
%
\begin{df}\label{def:alm_(3,5)}
  An \emph{almost $(3,5)$-distribution\/} on a $5$-dimensional manifold $M$ 
is defined as a tuple $(\D,\{\omega_1,\omega_2\})$, 
where $\D\subset TM$ is a tangent distribution of rank~$3$ on $M$ 
and $\{\omega_1,\omega_2\}$ is a pair of $2$-forms 
that are pointwise linearly independent on $\D$. 
\end{df}
\noindent
 Note that in a precise sense, the $2$-forms take values in $TM/\D$, 
due to orientations. 
 We remark that a ``genuine'' $(3,5)$-distribution 
is an almost $(3,5)$-distribution, in some sense. 
 Indeed, if $\D\subset TM$ is a genuine $(3,5)$-distribution 
on a $5$-dimensional manifold $M$, 
from Proposition~\ref{prop:dbasis}, 
there locally exists a pair $(\alpha_1,\ \alpha_2)$ of $1$-forms 
that defines $\D$ as $\D=\{\alpha_1=0,\ \alpha_2=0\}$ satisfying 
that the following $4$-forms are pointwise linearly independent: 
\begin{equation*}
  \alpha_1\wedge\alpha_2\wedge d\alpha_1,\qquad 
  \alpha_1\wedge\alpha_2\wedge d\alpha_2. 
\end{equation*}
 This implies the $2$-forms $d\alpha_1$ and $d\alpha_2$ determine 
two $(TM/\D)$-valued $2$-forms $[d\alpha_1], [d\alpha_2]$ 
that  make $\D$ an almost $(3,5)$-distribution 
$(\D,\{[d\alpha_1],[d\alpha_2]\})$. 
 We also should remark that the distribution $\mathcal{E}$ of rank~$3$ 
in the definition of the almost Cartan structure 
(see Definition~\ref{def:alm_cartan}) is an almost $(3,5)$-distribution. 

  Then the existence and classification results for $(3,5)$-distributions 
are described as follows. 
%
%
\begin{thrm}\label{thm:thm_(3,5)}
  Let $M$ be a possibly closed $5$-dimensional manifold. 

\noindent 
\textup{(1)}\  The manifold $M$ admits a $(3,5)$-distribution 
  if and only if it admits an almost $(3,5)$-distribution. 

\noindent 
\textup{(2)}\  Let $\mathcal{D}_0,\ \mathcal{D}_1\subset TM$ 
be $(3,5)$-distributions on a $5$-dimensional manifold $M$. 
 If the associated almost $(3,5)$-distributions of $\D_0$ and $\D_1$ 
are homotopic through almost $(3,5)$-distributions, 
then $\D_0$ and $\D_1$ are homotopic through genuine $(3,5)$-distributions. 
\end{thrm} 
\noindent 
 This section is devoted to show this theorem. 

  For the proof of Theorem~\ref{thm:thm_(3,5)}, 
we reformulate the problem in terms of the \textit{h}-principle. 
 In order to describe the claim, we define the set of ``formal'' structures 
compared with that of ``genuine'' $(3,5)$-distributions. 
 Let $M$ be a possibly closed $5$-dimensional manifold. 
 Let $\Omega_{(3,5)}(M)$ denote a set of genuine $(3,5)$-distributions on $M$, 
and $\almtf$ a set of almost $(3,5)$-distributions on $M$. 
 As we mentioned after Definition~\ref{def:alm_(3,5)}, 
associated with any genuine $(3,5)$-distribution. 
there exists an almost $(3,5)$-distribution. 
 Then setting $\bar{\Omega}_{(3,5)}(M)$ 
as the set of almost $(3,5)$-distributions
that are associated with genuine $(3,5)$-distributions, 
we have $\bar{\Omega}_{(3,5)}(M)\subset\almtf$. 
 Using the notion, the key theorem for the proofs 
of Theorem~\ref{thm:thm_(3,5)} is formulated as follows. 

%
%
\begin{thrm}\label{thm:key_(3,5)}
Let $M$ be a possibly closed $5$-dimensional manifold. 

\noindent 
\textup{(1)}\  Let $(\D,\{\omega_1,\omega_2\})\in\almtf$ 
  be an almost $(3,5)$-distribution on $M$. 
  Then $(\D,\{\omega_1,\omega_2\})$ is homotopic in $\almtf$ 
  to an almost $(3,5)$-distribution associated with 
  a genuine $(3,5)$-distribution $\mathcal{E}\in\Omega_{(3,5)}(M)$ on $M$.

\noindent 
\textup{(2)}\  Let $\D_0,\D_1\in\Omega_{(3,5)}(M)$
  be genuine $(3,5)$-distributions on $M$, 
  and let $(\D_t,\{\omega_1^t,\omega_2^t\})\in\almtf$, $t\in[0,1]$, 
  be a path of almost $(3,5)$-distributions in $\almtf$ 
  between the associated almost $(3,5)$-distributions of $\D_0$ and $\D_1$. 
  Then $(\D_t,\{\omega_1^t,\omega_2^t\})$ can be deformed in $\almtf$ 
  to a path $(\tilde{\D}_t,\{[d\tilde{\alpha}_1^t],[d\tilde{\alpha}_2^t]\})
  \in\bar{\Omega}_{(3,5)}(M)$ 
  of almost $(3,5)$-distributions associated with genuine $(3,5)$-distributions 
  $\tilde{\D}_t\in\Omega_{(3,5)}(M)$, 
  while keeping the endpoints fixed: 
  $(\D_i,\{[d\alpha_1^i],[d\alpha_2^i]\})
    =(\tilde{\D}_i,\{[d\tilde{\alpha}_1^i],[d\tilde{\alpha}_2^i]\})$, $i=0,1$. 
\end{thrm} 
\noindent
 It is clear that Theorem~\ref{thm:thm_(3,5)} 
follows from Theorem~\ref{thm:key_(3,5)}. 
\begin{proof}[Proof of Theorem~\textup{\ref{thm:thm_(3,5)}}]
  (1)\ As was noted after Definition~\ref{def:alm_(3,5)}, 
  to each ``genuine'' $(3,5)$-distribution 
  there exists an almost $(3,5)$-distribution associated with it. 
  Then it remains to prove the converse. 
  Suppose that there exists an almost $(3,5)$-distribution 
  $(\D,\{\omega_1,\omega_2\})$ on a $5$-dimensional manifold $M$. 
  Then, by Theorem~\ref{thm:key_(3,5)}, 
  $(\D,\{\omega_1,\omega_2\})$ is homotopic in $\almtf$ 
  to an almost $(3,5)$-distribution 
  associated with a genuine $(3,5)$-distribution $\tilde{\D}$. 
  That is, there exists a genuine $(3,5)$-distribution $\tilde{\D}$ 
  on the manifold $M$. \\
  (2)\ Let $\D_0,\ \D_1\in\Omega_{(3,5)}(M)$ be genuine $(3,5)$-distributions. 
  Suppose that there exists a homotopy 
  $(\D_t,\{\omega_1^t,\omega_2^t\})\in\almtf$, $t\in[0,1]$, 
  between two almost $(3,5)$-distributions 
  $(\D_i,\{\omega_1^i,\omega_2^i\})=(\D_i,\{[d\alpha_1^i],[d\alpha_2^i]\})
  \in\bar{\Omega}_{(3,5)}(M)$  
  associated with the given genuine $(3,5)$-distribution 
  $\D_i$, for $i=0,1$. 
  Then, by Theorem~\ref{thm:key_(3,5)}, there exists a path 
  $(\tilde{\D}_t,\{[d\tilde{\alpha}_1^t],[d\tilde{\alpha}_2^t]\})
  \in\bar{\Omega}_{(3,5)}(M)$ 
  of almost $(3,5)$-distributions associated with genuine $(3,5)$-distributions 
  $\tilde{\D}_t\in\Omega_{(3,5)}(M)$, where $\tilde{\D}_i=\D_i$, $i=0,1$. 
  That is, there exists a homotopy 
  $\tilde{\D}_t\in\Omega_{(3,5)}(M)$, $t\in[0,1]$, 
  of genuine $(3,5)$-distributions 
  between $\D_0=\tilde{\D}_0$ and $\D_1=\tilde{\D}_1$. 
\end{proof}

  Theorem~\ref{thm:key_(3,5)} is proved 
in Subsection~\ref{sec:proof_(3,5)}. 
 It is proved 
from the viewpoint of the \textit{h}-principles 
introduced in Section~\ref{sec:h-prin}. 
 Gromov's convex integration method is applied. 
 For the method, we will determine the relevant differential relation 
and show its ampleness in the next subsection. 

\subsection{Differential relation and its ampleness}\label{sec:key_(3,5)}
 In this subsection, we determine the relevant differential relation 
and show its ampleness. 

  In order to deal with tangent distributions, 
we represent them by their coframings. 
 In addition, almost $(3,5)$-distribution is defined as a certain tuple 
of a tangent distribution and two $2$-forms. 
 To describe them, we first introduce 
two vector bundles and a first-order differential operator. 
 Let $M$ be a $5$-dimensional manifold, 
and $\D$ a tangent distribution of rank~$3$ on $M$. 
 In other words, it is of corank~$2$. 
 Then it is locally regarded as a kernel of two $1$-forms. 
 For the description, let $X_1\to M$ be the vector bundle over $M$ defined as 
\begin{equation*}\label{eq:Xbdl_(3,5)}
    X_1:=\bigoplus^2 T^\ast M\to M. 
\end{equation*}
 Then a coframing of the distribution $\D$ is locally regarded 
as a section of this bundle $X_1$. 
 In addition, we recognize almost $(3,5)$-distributions 
in a similar manner. 
 From Definition~\ref{def:alm_cartan}, an almost $(3,5)$-distribution 
is defined as a triple that consists of a tangent distribution of rank~$3$ 
and two $2$-forms.
 Then it is regarded as a pair of $1$-forms and a pair of $2$-forms. 
 For the description, let $Z_1\to M$ be the vector bundle over $M$ defined as 
\begin{equation*}\label{eq:Zbdl_(3,5)}
  Z_1:=\bigoplus^2\lft( T^\ast M\oplus\bigwedge^2 T^\ast M\rgt)\to M. 
\end{equation*}
 Then a tuple that consists of a tangent distribution of corank~$2$ 
and two $2$-forms is regarded as a section of this bundle $Z_1$. 
 Next, we define a first-order linear differential operator 
connecting two bundles introduced above. 
 Let $\mathcal{F}_1\colon\sect{X_1}\to\sect{Z_1}$ 
be the first order linear differential operator defined as  
\begin{equation}\label{eq:def_f_(3,5)}
  \mathcal{F}_1\colon
  (\alpha_1,\ \alpha_2)\mapsto 
  (\alpha_1,\ \alpha_2,\ d\alpha_1,\ d\alpha_2). 
\end{equation}
 We remark that the symbol $F_1\colon X_1^{(1)}\to Z_1$ 
of $\mathcal{F}_1$ is fiberwise epimorphic. 
 Then we can apply Theorem~\ref{thm:HP_D-secs} to this operator. 

  Using these descriptions, we formulate the problem as follows. 
 The hypotheses of Theorem~\ref{thm:key_(3,5)} are concerned 
with formal structures, that is, almost $(3,5)$-distributions. 
 Then we find the singular locus in $Z_1$, and by Theorem~\ref{thm:HP_D-secs}, 
we determine the differential relation to consider in the $1$-jet space 
$\ojet{(X_1)}$. 

  Now, we define the singular locus in $Z_1$, the target side 
of $\mathcal{F}_1$. 
 Let $\tilde{S}\subset Z_1$ be the subset defined as
\begin{align}\label{eq:def_s_(3,5)}
  \tilde{S}:=&\lft\{(\alpha_1,\alpha_2,\omega_1,\omega_2)_p
         \in\bigoplus^2\lft(T^\ast M\oplus\bigwedge^2T^\ast M\rgt)=Z_1 \rgt. \\
       &\qquad \Big|\; \left. \vphantom{\bigoplus^2}
         (\alpha_1\wedge\alpha_2\wedge\omega_1)_p\quad \text{and}\quad
         (\alpha_1\wedge\alpha_2\wedge\omega_2)_p 
         \quad \text{are linearly dependent}\rgt\}. \notag
\end{align}
 It is defined in the fiber $(Z_1)_p$ by the condition that 
$(\alpha_1\wedge\alpha_2\wedge\omega_1)_p$ and 
$(\alpha_1\wedge\alpha_2\wedge\omega_2)_p$ are linearly dependent. 
 In other words, it is where the defining condition 
of almost $(3,5)$-distribution does not hold. 

 Then we take the inverse image of $\tilde{S}\subset Z_1$ 
by the symbol $F_1\colon X_1^{(1)}\to Z_1$ 
of $\mathcal{F}_1\colon \sect {X_1}\to \sect{Z_1}$ 
(see Section~\ref{sec:h-prin} for definition). 
 Set 
\begin{equation}\label{eq:def_sgm1}
  \tilde{\Sigma}:=(F_1)^{-1}(\tilde{S})\subset X_1^{(1)}. 
\end{equation}
 Then, since the subset $\tilde{S}\subset Z_1$ corresponds to 
where the defining condition of almost $(3,5)$-distribution does not hold,
the subset $\tilde{\Sigma}=(F_1)^{-1}(\tilde{S})\subset X_1^{(1)}$ is regarded 
as the singularity to be considered. 
 In other words, it corresponds the locus where the condition 
for genuine $(3,5)$-distribution does not hold. 
 Let $\tilde{\mathcal{R}}\subset \ojet{X_1}$ denote the complement 
of the singularity $\tilde{\Sigma}$: 
\begin{equation}\label{eq:def_dr_35}
\tilde{\mathcal{R}}:=X_1^{(1)}\setminus\tilde{\Sigma}\subset X_1^{(1)}. 
\end{equation}
 It is the open differential relation to be considered 
in order to show Theorem~\ref{thm:key_(3,5)}. 

  We will show that the differential relation 
$\tilde{\mathcal{R}}\subset \ojet{X_1}$ satisfies the \textit{h}-principles. 
 To this end, Gromov's convex integration method is applied. 
 The essence of the proof is the following proposition. 
%
%
\begin{prop}\label{prop:essence_(3,5)}
  The open differential relation $\tilde{\mathcal{R}}\subset \ojet{X_1}$ 
is ample. 
\end{prop} 

\begin{proof}
  It is sufficient to discuss the problem using local coordinates 
of the base manifold $M$ of the bundles. 
 Let $(x_1,x_2,\dots,x_5)$ be local coordinates 
of the $5$-dimensional manifold $M$. 
 Then $\{(dx_1)_p,\dots,(dx_5)_p\}$ is a basis of the fiber $T^\ast_pM$. 
 Let $(a_1,\dots,a_5)$ be coordinates on the fiber of $T^\ast M$. 
 Similarly, $\{(dx_i\wedge dx_j)_p\}_{1\le i<j\le 5}$ is a basis of the fiber 
$\lft(\bigwedge^2 T^\ast M\rgt)_p$. 
 Let $(z_{12},\dots,z_{45})$ be coordinates on the fiber 
of $\bigwedge^2 T^\ast M$. 
 Further, let $(a_1,\dots,a_5,y_{11},y_{12},\dots,y_{55})$ be coordinates of 
the fiber $(T^\ast M)^{(1)}_v$, 
where $y_{ij}$ corresponds to $\rd a_i/\rd x_j$. 

  By using such local coordinates, we write down the singularity 
$\tilde{\Sigma}\subset \ojet{\lft(\bigoplus^2 T^\ast M\rgt)}=\ojet{X_1}$ as follows. 
 Recall that $\tilde{\Sigma}$ is defined 
as the inverse image $\tilde{\Sigma}={F_1}^{-1}(\tilde{S})$ 
by the symbol $F_1\colon\ojet{X_1}\to\sect{Z_1}$ 
of the linear differential operator 
$\mathcal{F}_1\colon\sect{X_1}\to\sect{Z_1}$ 
defined by Equation~\eqref{eq:def_f_(3,5)} concerning exterior derivative 
of differential forms. 
 Note that, by the exterior derivative, 
the coordinates $y_{ij}-y_{ji}$ correspond 
to the anti-symmetric components $z_{ij}$. 
 Recall that the set $\tilde{S}\subset\ojet{Z_1}$ 
is defined by the following condition (see Equation~\eqref{eq:def_s_(3,5)}): 
two $4$-forms 
\begin{equation}\label{eq:dep_forms_(3,5)}
  \alpha_1\wedge\alpha_2\wedge\omega_1, \quad 
  \alpha_1\wedge\alpha_2\wedge\omega_2
\end{equation}
are linearly independent on each fiber, 
where $\alpha_i\in\sect{T^\ast M}$ are $1$-forms 
and $\omega_i\in\sect{\bigwedge^2 T^\ast M}$ are $2$-forms on $M$. 
 Then, in order to write down the singularity 
$\tilde{\Sigma}={F_1}^{-1}(\tilde{S})\subset \ojet{X}$ by using local coordinates, 
we write down $\alpha_i$ and $\omega_i$ on a fiber over $p\in M$ 
by such coordinates as follows: 
\begin{align*}
  \alpha_i&=\sum_{j=1}^5 a^i_jdx_j 
            =a^i_1dx_1+a^i_2dx_2+\dots+a^i_5dx_5, \quad (i=1,2), \\
  \omega_i&=\sum_{1\le j<k\le 5}z^i_{jk}dx_j\wedge dx_k 
          =z^i_{12}dx_1\wedge dx_2+\dots+z^i_{45}dx_{4}\wedge dx_5, 
            \qquad (i=1,2). 
\end{align*}
 Following this representation, the $4$-forms 
in Equation~\eqref{eq:dep_forms_(3,5)} are written down as follows. 
 First, we have 
\begin{equation*}
  \alpha_1\wedge\alpha_2=\sum_{1\le i<j\le 5}A_{ij}dx_i\wedge dx_j, 
\end{equation*}
where $A_{ij}$ are minor determinants 
\begin{equation*}
  A_{ij}:=
  \begin{vmatrix}
    a^1_i & a^1_j\\ a^2_i & a^2_j 
  \end{vmatrix},\quad 1\le i<j\le 5. 
\end{equation*}
 On the other hand, $z^i_{jk}$ is valid for $j<k$. 
 Then the $4$-forms in Equation~\eqref{eq:dep_forms_(3,5)} are 
\begin{equation}\label{eq:loc_dep_forms}
  \alpha_1\wedge\alpha_2\wedge\omega_i
  =\sum_{r=1}^5(B^i_r) dx_1\wedge\dots\wedge\widehat{dx_r}\wedge
     \dots\wedge dx_5, 
   \qquad i=1,2, 
\end{equation}
where ``$\widehat{dx_r}$'' implies ``without $dx_r$'', and 
the coefficients are 
\begin{equation}\label{eq:def_Bij}
  \begin{aligned}
    B^i_1&=A_{23}z^i_{45}-A_{24}z^i_{35}+A_{25}z^i_{34}
           +A_{34}z^i_{25}-A_{35}z^i_{24}+A_{45}z^i_{23}, \\
    B^i_2&=A_{13}z^i_{45}-A_{14}z^i_{35}+A_{15}z^i_{34}
           +A_{34}z^i_{15}-A_{35}z^i_{14}+A_{45}z^i_{13}, \\
    B^i_3&=A_{12}z^i_{45}-A_{14}z^i_{25}+A_{15}z^i_{24}
           +A_{24}z^i_{15}-A_{25}z^i_{14}+A_{45}z^i_{12}, \\
    B^i_4&=A_{12}z^i_{35}-A_{13}z^i_{25}+A_{15}z^i_{23}
           +A_{23}z^i_{15}-A_{25}z^i_{13}+A_{35}z^i_{12}, \\
    B^i_5&=A_{12}z^i_{34}-A_{13}z^i_{24}+A_{14}z^i_{23}
           +A_{23}z^i_{14}-A_{24}z^i_{13}+A_{34}z^i_{12}, 
  \end{aligned}\qquad i=1,2. 
\end{equation}
 Then, the description of the singularity 
$\tilde{\Sigma}\subset\ojet{X_1}$ by the local coordinates 
is obtained by using the description in Equation~\eqref{eq:loc_dep_forms} 
for the $4$-forms in Equation~\eqref{eq:dep_forms_(3,5)}. 
 Recall that $\tilde{S}\subset\ojet{Z_1}$ is the locus 
where two $4$-forms are linearly dependent. 
 This implies the following two vectors are linearly dependent: 
\begin{equation*}
  (B^1_1,B^1_2,\dots,B^1_5),\qquad (B^2_1,B^2_2,\dots,B^2_5). 
\end{equation*}
 In other words, the following system of equations holds: 
\begin{equation}\label{eq:sys_for_sing}
  \begin{aligned}
    F_{12}&:=
            \begin{vmatrix}
              B^1_1& B^1_2 \\ B^2_1& B^2_2
            \end{vmatrix}=0,\quad
            F_{13}:=
            \begin{vmatrix}
              B^1_1& B^1_3 \\ B^2_1& B^2_3
            \end{vmatrix}=0,\dots,\quad 
            F_{15}:=
            \begin{vmatrix}
              B^1_1& B^1_5 \\ B^2_1& B^2_5
            \end{vmatrix}=0, \\
    F_{23}&:=
            \begin{vmatrix}
              B^1_2& B^1_3 \\ B^2_2& B^2_3
            \end{vmatrix}=0,\dots,\quad 
            F_{34}:=
            \begin{vmatrix}
              B^1_3& B^1_4 \\ B^2_3& B^2_4
            \end{vmatrix}=0,\dots,\quad 
            F_{45}:=
            \begin{vmatrix}
              B^1_4& B^1_5 \\ B^2_4& B^2_5
            \end{vmatrix}=0.
  \end{aligned}
\end{equation}
 Note that this system consists of $10={}_5C_2$ equations, 
because the column numbers of the $2\times 2$ matrices 
in Equations~\eqref{eq:sys_for_sing}
are combinations of $1,2,\dots,5$. 
 This is the local description of the singularity $\tilde{S}\subset\ojet{Z_1}$. 
 In other words, the system of Equations~\eqref{eq:sys_for_sing} 
determines the singularity $\tilde{S}\subset\ojet{X_1}$. 
 Recall that the singularity $\tilde{\Sigma}\subset\ojet{X_1}$ is defined 
as $\tilde{\Sigma}={F_1}^{-1}(\tilde{S})$ (see Equation~\eqref{eq:def_sgm1}). 

  In order to show the ampleness of the differential relation 
$\tilde{\mathcal{R}}=\ojet{X_1}\setminus \tilde{\Sigma}$, 
we should observe the intersection of $\tilde{\mathcal{R}}\subset\ojet{X_1}$ 
or $\tilde{\Sigma}\subset\ojet{X_1}$ with principal subspaces 
(see Section~\ref{sec:cvx_itgr} for definition). 
 To make the discussion simple, we take a principal direction $P_1$ 
that concerns $x_1$ in the local coordinates of the base manifold $M$. 
 In other words, the intersection of $\tilde{\Sigma}$ and $P_1$ 
is determined by the system of equations (Equations~\eqref{eq:sys_for_sing}) 
with variables only $z^i_{1l}$, $l=1,2,\dots,5$. 
 Other $z^i_{ml}$ and $a^i_j$ are constant on the intersection 
in a fiber over $p\in M$. 
 Because, $\tilde{\Sigma}\subset\ojet{X_1}$ is defined as $\tilde{\Sigma}={F_1}^{-1}(\tilde{S})$ 
by using the linear differential operator 
$\mathcal{F}_1\colon\sect{X_1}\to\sect{Z_1}$
defined by exterior derivative of differential forms 
(see Equations~\eqref{eq:def_f_(3,5)} and~\eqref{eq:def_sgm1}). 
 We use the same symbol $B^i_j$ for the intersection. 

  We now show that $\tilde{\Sigma}\cap P_1\subset P_1$ is ample. 
 The discussion is divided into two cases. 
 Case~1 is when $(B^1_1,B^2_1)\ne (0,0)$, 
and Case~2 is when $(B^1_1,B^2_1)=(0,0)$. 
 As we mentioned above, now Equations~\eqref{eq:sys_for_sing} are considered  
on the principal subspace $P_1\subset\ojet{X_1}$. 
 In other words, only $z^i_{1l}$ are variables. 
 From the definition of $B^i_j$ (see Equations~\eqref{eq:def_Bij}), 
$B^i_1$, $i=1,2$, have no variables if a fiber is fixed. 
 The two cases above depend on the choices of a base point. 

\noindent\underline{\textbf{Case 1.}}\ 
 Assume that $(B^1_1,B^2_1)\ne (0,0)$. 
 Then the system of equations in Equations~\eqref{eq:sys_for_sing} 
is reduced to that of the first four equations: 
\begin{equation*}
  F_{12}=0,\quad F_{13}=0,\quad F_{14}=0,\quad F_{15}=0. 
\end{equation*}
 In fact, if these hold, other equations hold automatically. 
 Recall that this system determines $\tilde{\Sigma}\cap P_1$. 
 It is regarded as the system of linear equations 
with respect to eight variables, 
$z^1_{12},z^1_{13},z^1_{14},z^1_{15},z^2_{12},z^2_{13},z^2_{14},z^2_{15}$. 
 They are actually written down as follows: 
\begin{equation}\label{eq:Case1}
  \lft\{
  \begin{aligned}
    (\ba2145)\zed113-&(\ba2135)\zed114+(\ba2134)\zed115
    -(\ba1145)\zed213+(\ba1135)\zed214-(\ba1134)\zed215\\
    &=B^1_1(A_{13}\zed245-A_{14}\zed235+A_{15}\zed234)
    -B^2_1(A_{13}\zed145-A_{14}\zed135+A_{15}\zed134), \\
    (\ba2145)\zed112-&(\ba2125)\zed114+(\ba2124)\zed115
    -(\ba1145)\zed212+(\ba1125)\zed214-(\ba1124)\zed215\\
    &=B^1_1(A_{12}\zed245-A_{14}\zed225+A_{15}\zed224)
    -B^2_1(A_{12}\zed145-A_{14}\zed125+A_{15}\zed124), \\
    (\ba2135)\zed112-&(\ba2125)\zed113+(\ba2123)\zed115
    -(\ba1135)\zed212+(\ba1125)\zed213-(\ba1123)\zed215\\
    &=B^1_1(A_{12}\zed235-A_{13}\zed225+A_{15}\zed223)
    -B^2_1(A_{12}\zed135-A_{13}\zed125+A_{15}\zed123), \\
    (\ba2134)\zed112-&(\ba2124)\zed113+(\ba2123)\zed114
    -(\ba1135)\zed212+(\ba1124)\zed213-(\ba1123)\zed214\\
    &=B^1_1(A_{12}\zed234-A_{13}\zed224+A_{14}\zed223)
    -B^2_1(A_{12}\zed134-A_{13}\zed124+A_{14}\zed123). 
  \end{aligned}
  \rgt.
\end{equation}
 Note that the right-hand sides are constants. 
 Then the coefficient matrix of the system~\eqref{eq:Case1} is
\begin{equation}\label{eq:matrix(3,5)}
  \begin{pmatrix}
    0&\ba2145&-\ba2135&\ba2134&0&-\ba1145&\ba1135&-\ba1134 \\
    \ba2145&0&-\ba2125&\ba2124&-\ba1145&0&\ba1125&-\ba1124 \\
    \ba2135&-\ba2125&0&\ba2123&-\ba1135&\ba1125&0&-\ba1123 \\
    \ba2134&-\ba2124&\ba2123&0&-\ba1134&\ba1124&-\ba1123&0 
  \end{pmatrix}. 
\end{equation}
 The rank of this matrix can never be~$2$. 
 In fact, it is a $4\times 8$~matrix 
each of whose first and last four rows form symmetric blocks up to sign: 
\begin{equation*}
  \begin{array}{llll} 
    a_{jj}=0,\quad     & a_{ij}=\pm a_{ji},      &\quad &(j\le 4), \\
    a_{(j-4)j}=0,\quad & a_{ij}=\pm a_{j-4,i+4}, &\quad &(j>4), 
  \end{array}
\end{equation*}
where $a_{ij}$ is an $(i,j)$-entry of the matrix. 
 Suppose that $a_{ij}\ne 0$ for $j\le 4$. 
 Take a minor determinant consists 
of the $i$-th and $j$-th rows and of the $i$-th and $j$-th columns. 
 Then we have 
\begin{equation*}
  \begin{vmatrix}
    a_{ii}& a_{ij}\\ a_{ji} & a_{jj}
  \end{vmatrix}
  =0\pm {a_{ij}}^2=\pm {a_{ij}}^2\ne 0. 
\end{equation*}
 When $a_{ij}\ne0$ for $j>4$, similarly we have 
\begin{equation*}
  \begin{vmatrix}
    a_{i\ (i+4)}& a_{ij}\\ a_{(j-4)\ (i+4)} & a_{(j-4)\ j}
  \end{vmatrix}
  =0\pm {a_{ij}}^2=\pm {a_{ij}}^2\ne 0. 
\end{equation*}
 Then the rank of the coefficient matrix~\eqref{eq:matrix(3,5)} above 
is greater than or equal to~$2$ unless the rank is not $0$. 
 When the rank is $0$, the singularity $\tilde{\Sigma}\cap P_1=P_1$ 
is of codimension~$0$ in the principal subspace. 
 This implies $\tilde{\mathcal{R}}\cap P_1=\emptyset$. 
 Then the differential relation $\tilde{\mathcal{R}}\subset\ojet{X_1}$ 
is ample by definition. 
 When the rank of the coefficient matrix~\eqref{eq:matrix(3,5)} above 
is greater than or equal to~$2$, 
the singularity $\tilde{\Sigma}\cap P_1=P_1$ is thin, that is, 
of codimension at least~$2$ in the principal subspace. 
 Hence the differential relation $\tilde{\mathcal{R}}\subset\ojet{X_1}$ 
is ample. 

\noindent\underline{\textbf{Case~2.}}\ 
  Assume that $(B^1_1,B^2_1)= (0,0)$. 
 Then the system of equations in Equations~\eqref{eq:sys_for_sing} 
is reduced to that of the latter six equations: 
\begin{equation*}
  F_{23}=0,\quad F_{24}=0,\quad F_{25}=0,\quad 
  F_{34}=0,\quad F_{35}=0,\quad F_{45}=0. 
\end{equation*}
 In fact, the first four equations hold immediately from the assumption. 
 Then the singularity $\tilde{\Sigma}\cap P_1$ is considered 
as an intersection of six hypersurfaces $\{F_{ij}=0\}$. 
 If the complement $[\{F_{ij}=0\}\cap P_1]^c=\{F_{ij}\ne 0\}\cap P_1$ 
of one of the hypersurfaces is ample in $P_1$, 
then the differential relation $\tilde{\mathcal{R}}\cap P_1$ is ample in $P_1$. 
 In the following, 
we first discuss the case when all hypersurfaces are linear, 
that is, the complement of each hyperplane is not ample. 
 Then, when at least one of them is not linear, say $F_{23}$, 
we will show that the complement $[\{F_{23}=0\}\cap P_1]^c$ is ample. 

  First, we write down $F_{23}$ as a polynomial with eight variables 
$\zed112, \zed113, \zed114, \zed115, \zed212, \zed213, \zed214, \zed215$. 
 The functions $B^i_j$ are defined in Equations~\eqref{eq:def_Bij}. 
 On the principal subspace $P_1$, they are rewritten as follows: 
\begin{align*}
    B^i_2&=A_{45}z^i_{13}-A_{35}z^i_{14}+A_{34}z^i_{15}+C^i_2, \\
    B^i_3&=A_{45}z^i_{12}-A_{25}z^i_{14}+A_{24}z^i_{15}+C^i_3,\quad i=1,2, \\
\end{align*}
where 
\begin{align*}
  C^i_2&:=A_{13}z^i_{45}-A_{14}z^i_{35}+A_{15}z^i_{34}, \\
  C^i_3&:=A_{12}z^i_{45}-A_{14}z^i_{25}+A_{15}z^i_{24},\quad i=1,2, 
\end{align*}
are constants. 
 Then we have 
\begin{align}\label{eq:f23}
  F_{23}=&
          \begin{vmatrix}
            B^1_2&B^1_3\\ B^2_2&B^2_3
          \end{vmatrix}
          =B^1_2B^2_3-B^1_3B^2_2 \\
  =&(-\aaa45^2)\zed112\zed213+(\aaa45\aaa35)\zed112\zed214
     +(-\aaa45\aaa34)\zed112\zed215 \notag \\
        &+(\aaa45^2)\zed113\zed212+(-\aaa45\aaa25)\zed113\zed214
          +(\aaa45\aaa24)\zed113\zed215 \notag \\
        &+(-\aaa35\aaa45)\zed114\zed212+(\aaa25\aaa45)\zed114\zed213
          +(\aaa25\aaa34-\aaa24\aaa35)\zed114\zed215 \notag \\
        &+(\aaa34\aaa45)\zed115\zed212+(-\aaa24\aaa45)\zed115\zed213
          +(\aaa24\aaa35-\aaa25\aaa34)\zed115\zed214 \notag \\
        &+(-\ca2245)\zed112+(\ca2345)\zed113+(\ca2225-\ca2335)\zed114
          +(\ca2334-\ca2224)\zed115 \notag \\
        &+(\ca1245)\zed212+(-\ca1345)\zed213+(\ca1335-\ca1225)\zed214
          +(\ca1224-\ca1334)\zed215 \notag \\
        &+(\cc12\cc23-\cc13\cc22). \notag
\end{align}

  If all the six functions $F_{23},\ F_{24},\dots, F_{45}$ are linear, 
we can deduce the differential relation $\tilde{\mathcal{R}}\subset\ojet{X_1}$ 
is ample as follows. 
  The function $F_{23}$ is linear if and only if 
\begin{equation}\label{eq:lin-hypsf}
  \lft\{
  \begin{aligned}
    &\aaa45=0, \\
    &\aaa24\aaa35-\aaa25\aaa34=0. 
  \end{aligned}
  \rgt.
\end{equation}
 Then, by the same discussion for other functions $F_{ij}$, 
if all six functions $F_{23},\ F_{24},\dots, F_{45}$ are linear, we have 
\begin{equation*}
  \aaa23=\aaa24=\aaa25=\aaa34=\aaa35=\aaa45=0. 
\end{equation*}
 This implies that all $F_{23},\ F_{24},\dots, F_{45}$ are constant functions. 
 Then the singularity $\tilde{\Sigma}\cap P_1$ is whole $P_1$ or empty $\emptyset$. 
 Therefore the differential relation $\tilde{\mathcal{R}}\subset\ojet{X_1}$ is ample. 

  In the following, we assume at least one of $F_{23}, F_{24},\dots, F_{45}$ 
is not linear. 
 Without loss of generality, we may assume $F_{23}$ is not linear. 
 In other words, Condition~\eqref{eq:lin-hypsf} does not hold. 

  Since $F_{23}$ is a polynomial of degree~$2$, 
we observe the normal form of the quadric hypersurface $F_{23}=0$. 
 Setting the matrices as follows, the function $F_{23}$ is represented as: 
\begin{equation*}
  F_{23}=\bs{z}^\mathsf{T}G\bs{z}+2H^\mathsf{T}\bs{z}+K  
\end{equation*}
(see Equation~\ref{eq:f23}). 
 The symbol $A^\mathsf{T}$ denotes the transpose of matrix $A$. 
 Set
\begin{equation*}
  \bs{z}:=
  \begin{pmatrix}
    \zed112&\zed113&\zed114&\zed115&\zed212&\zed213&\zed214&\zed215
  \end{pmatrix}^\mathsf{T}, 
\end{equation*}
\begin{equation*}
  G:=\frac{1}{2}
  \begin{pmatrix}
    \bs{O} & G_1\\ {G_1}^\mathsf{T} & \bs{O}
  \end{pmatrix}, 
\end{equation*} where 
\begin{equation*}
  G_1:=
  \begin{pmatrix}
    0&-\aaa45^2&\aaa35\aaa45&-\aaa34\aaa45 \\
    \aaa45^2&0&-\aaa25\aaa45&\aaa24\aaa45 \\
    -\aaa35\aaa45&\aaa25\aaa45&0&(\aaa25\aaa34-\aaa24\aaa35) \\
    \aaa34\aaa45&-\aaa24\aaa45&(\aaa24\aaa35-\aaa25\aaa34)&0
  \end{pmatrix}
\end{equation*}
and $\bs{O}$ is the $4\times4$ zero matrix, 
\begin{align*}
  H:= &\lft( -\ca2245\quad \ca2345\quad (\ca2225-\ca2335)\quad 
       (\ca2334-\ca2224)\rgt.  \\
     &\quad \lft. \ca1245\quad -\ca1345\quad (\ca1335-\ca1225)\quad 
       (\ca1224-\ca1334)\rgt) ^\mathsf{T}, 
\end{align*} 
and 
\begin{equation*}
  K:=\cc12\cc23-\cc13\cc22. 
\end{equation*}
 The matrix $G$ corresponding to degree~$2$ part is a symmetric matrix. 
 Since it is diagonalizable by linear transformation, 
then we obtain the normal form of the quadric hypersurface $\{F_{23}=0\}$. 
 The eigenvalues of $G$ are $\gamma$, $-\gamma$ and $0$, where 
\begin{equation*}
  \gamma:=\frac{1}{2}
  \sqrt{\aaa45^2(\aaa23^2+\aaa25^2+\aaa34^2+\aaa35^2+\aaa45^2)
    +(\aaa24\aaa35-\aaa25\aaa34)^2}, 
\end{equation*}
with respective multiplicities $2$, $2$, and $4$. 
 Recall that we now assume that at least one of Equations~\eqref{eq:lin-hypsf} 
does not hold. 
 Then it follows that $\gamma\ne 0$. 
 In other words, $G$ has positive eigenvalue $\gamma>0$ of multiplicity~$2$ 
and negative eigenvalue $-\gamma<0$ of multiplicity~$2$. 
 Then, by appropriate affine transformations, 
the quadric hypersurface $\{F_{23}=0\}$ is reduced 
to one of the following normal forms in the $(x_1,\dots,x_8)$-space: 
\begin{enumerate}
\item ${x_1}^2+{x_2}^2-{x_3}^2-{x_4}^2=c$,\quad ``central'' type, 
\item ${x_1}^2+{x_2}^2-{x_3}^2-{x_4}^2+x_5=0$, \quad ``non-central'' type. 
\end{enumerate}

  Since the transformations are affine, 
we can discuss the ampleness of the complement 
of the quadric hypersurface $\{F_{23}=0\}$ using the normal forms above. 

\noindent
\underline{(1)\ central quadrics.}\quad 
In $\mathbb{R}^8$ with coordinates $(x_1,x_2,\dots,x_8)$, 
the function $F_{23}$ is reduced to $F_{23}={x_1}^2+{x_2}^2-{x_3}^2-{x_4}^2-c$. 
 We keep the same notation $F_{23}$ for this normal form. 
 The complement of $\{F_{23}=0\}$ is divided 
into the following two path-connected components: 
\begin{equation*}
  U_\pm=\lft\{(x_1,\dots,x_8)\in\mathbb{R}^8\mid 
  {x_1}^2+{x_2}^2-{x_3}^2-{x_4}^2\gtrless c\rgt\}. 
\end{equation*}
 We will show that $\cv{U_\pm}=\mathbb{R}^8$, 
where $\cv{U_\pm}$ denotes the convex hull of $U_{\pm}$. 
 This case is further divided into two cases: 
(i)\ $c=0$, (``cone'' type), (ii)\ $c\ne 0$, (``hyperboloid" type). 

\noindent Case~(i): $c=0$.\ 
 Now, $F_{23}={x_1}^2+{x_2}^2-{x_3}^2-{x_4}^2$. 
 First, we show $\cv{U_+}=\mathbb{R}^8$. 
 With an arbitrary point $p=(p_1,p_2,\dots,p_8)\in\mathbb{R}^8\setminus U_+$, 
we show $p\in\cv{U_+}$. 
 From the assumption, we have 
\begin{equation*}
  {p_1}^2+{p_2}^2-{p_3}^2-{p_4}^2\le 0. 
\end{equation*}
 Then by setting 
\begin{align*}
  p_+^1&:=\lft(\sqrt{{p_3}^2+{p_4}^2-{p_2}^2+\epsilon},p_2,p_3,p_4,
  \dots,p_8\rgt), \\ 
  p_+^2&:=\lft(-\sqrt{{p_3}^2+{p_4}^2-{p_2}^2+\epsilon},p_2,p_3,p_4,
  \dots,p_8\rgt),
\end{align*}
for $\epsilon>0$, we have two points $p_+^1,p_+^2\in U_+$ 
(see Figure~\ref{fig:cone}). 
%
%
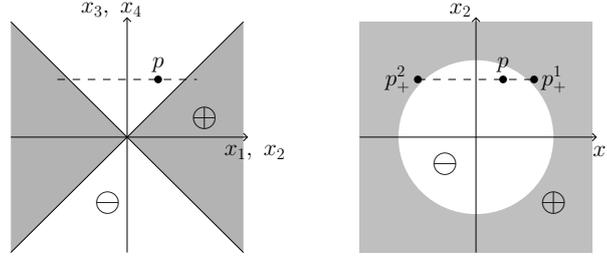
\begin{figure}[htb]
  \centering

\definecolor{grey}{RGB}{128,128,128}
\definecolor{cb3b3b3}{RGB}{179,179,179}

\def \globalscale {0.5100000}
\begin{tikzpicture}[y=1cm, x=1cm, yscale=\globalscale,xscale=\globalscale, every node/.append style={scale=\globalscale}, inner sep=0pt, outer sep=0pt]
  \path[fill=grey,opacity=0.5,line cap=butt,line join=miter,line 
  width=0.0cm,miter limit=4.0] (11.0, 27.7) -- (11.0, 21.7) -- (17.0, 21.7) -- 
  (17.0, 27.7) -- cycle;

  \path[fill=white,line width=0.0cm] (14.0, 24.7) ellipse (2.0cm and 2.0cm);

  \path[fill=cb3b3b3,even odd rule,draw opacity=0.0,line cap=butt,line 
  join=miter,line width=0.0cm] (8.0, 27.7) -- (5.0, 24.7) -- (8.0, 21.7);

  \path[fill=cb3b3b3,even odd rule,draw opacity=0.0,line cap=butt,line 
  join=miter,line width=0.0cm] (2.0, 27.7) -- (5.0, 24.7) -- (2.0, 21.7);

  \path[draw=black,even odd rule,line cap=butt,line join=miter,line width=0.0cm]
   (5.0, 27.8) -- (5.0, 21.7);

  \path[draw=black,even odd rule,line cap=butt,line join=miter,line width=0.0cm]
   (2.0, 21.7) -- (8.0, 27.7);

  \path[draw=black,even odd rule,line cap=butt,line join=miter,line width=0.0cm]
   (2.0, 27.7) -- (8.0, 21.7);

  \path[draw=black,even odd rule,line cap=butt,line join=miter,line width=0.0cm]
   (2.0, 24.7) -- (8.1, 24.7);

  \path[draw=black,line cap=butt,line join=miter,line width=0.0cm,miter 
  limit=4.0] (4.9, 27.7) -- (5.0, 27.8) -- (5.1, 27.7);

  \path[draw=black,line cap=butt,line join=miter,line width=0.0cm,miter 
  limit=4.0] (8.0, 24.8) -- (8.1, 24.7) -- (8.0, 24.6);

  \node at (8.3, 24.3) {\Large $x_1,\ x_2$};

  \node at (4.6, 28.0) {\Large $x_3,\ x_4$};

  \path[draw=black,even odd rule,line cap=butt,line join=miter,line width=0.0cm]
   (14.0, 27.8) -- (14.0, 21.7);

  \path[draw=black,even odd rule,line cap=butt,line join=miter,line width=0.0cm]
   (11.0, 24.7) -- (17.1, 24.7);

  \path[draw=black,line cap=butt,line join=miter,line width=0.0cm,miter 
  limit=4.0] (13.9, 27.7) -- (14.0, 27.8) -- (14.1, 27.7);

  \path[draw=black,line cap=butt,line join=miter,line width=0.0cm,miter 
  limit=4.0] (17.0, 24.8) -- (17.1, 24.7) -- (17.0, 24.6);

  \node at (17.3, 24.3) {\Large $x_1$};

  \node at (13.6, 28.0) {\Large $x_2$};

  \path[draw=black,line cap=butt,line join=miter,line width=0.0cm,miter 
  limit=4.0,dash pattern=on 0.1cm off 0.1cm] (3.2, 26.2) -- (6.8, 26.2);

  \path[fill=black,line width=0.0cm,dash pattern=on 0.1cm off 0.1cm] (12.5, 
  26.2) ellipse (0.1cm and 0.1cm) node [left, xshift=-0.2cm]{\Large $p_+^2$};



  \path[fill=black,line width=0.0cm,dash pattern=on 0.1cm off 0.1cm] (5.8, 26.2)
   ellipse (0.1cm and 0.1cm) node[above, yshift=0.2cm]{\Large $p$};

  \path[draw=black,line cap=butt,line join=miter,line width=0.0cm,miter 
  limit=4.0,dash pattern=on 0.1cm off 0.1cm] (12.5, 26.2) -- (15.5, 26.2);

  \path[fill=black,line width=0.0cm,dash pattern=on 0.1cm off 0.1cm] (15.5, 
  26.2) ellipse (0.1cm and 0.1cm) node[right, xshift=0.2cm]{\Large $p_+^1$};

  \path[fill=black,line width=0.0cm,dash pattern=on 0.1cm off 0.1cm] (14.7, 
  26.2) ellipse (0.1cm and 0.1cm) node[above, yshift=0.2cm]{\Large $p$};

  \node at (4.5, 23) {\Huge $\ominus$}; 
  \node at (7, 25.2) {\Huge $\oplus$}; 
  \node at (13.2, 24) {\Huge $\ominus$}; 
  \node at (16, 23) {\Huge $\oplus$};

\end{tikzpicture}
  \caption{cone type}
  \label{fig:cone}
\end{figure} 
 The point $p\in \mathbb{R}^8\setminus U_+$ 
lies on the line segment $p_+^1p_+^2$, 
and divides it in the ratio 
$\sqrt{{p_3}^2+{p_4}^2-{p_2}^2+\epsilon}-p_1 : 
p_1+\sqrt{{p_3}^2+{p_4}^2-{p_2}^2+\epsilon}$ internally. 
 In other words, $p\in\cv{U_+}$. 
 Then it follows $\cv{U_+}=\mathbb{R}^8$. 

  We can show $\cv{U_-}=\mathbb{R}^8$ in a similar way 
by exchanging the roles of $x_1,\ x_2$ and $x_3,\ x_4$. 

\noindent Case~(ii): $c\ne 0$.\ 
 Now, $F_{23}={x_1}^2+{x_2}^2-{x_3}^2-{x_4}^2-c$. 
 We may assume $c>0$. 
 First, we show $\cv{U_+}=\mathbb{R}^8$. 
 With an arbitrary point $p=(p_1,p_2,\dots,p_8)\in\mathbb{R}^8\setminus U_+$, 
we show $p\in\cv{U_+}$. 
 From the assumption, we have 
\begin{equation*}
  {p_1}^2+{p_2}^2-{p_3}^2-{p_4}^2-c\le 0. 
\end{equation*}
 Then by setting 
\begin{align*}
  p_+^1&:=\lft(\sqrt{{p_3}^2+{p_4}^2-{p_2}^2+c+\epsilon},p_2,p_3,p_4,
  \dots,p_8\rgt), \\ 
  p_+^2&:=\lft(-\sqrt{{p_3}^2+{p_4}^2-{p_2}^2+c+\epsilon},p_2,p_3,p_4,
  \dots,p_8\rgt),
\end{align*}
for $\epsilon>0$, we have two points $p_+^1,p_+^2\in U_+$ 
(see Figure~\ref{fig:cent-h+}). 
%
%
\begin{figure}[htb]
  \centering

\definecolor{grey}{RGB}{128,128,128}

\def \globalscale {0.51000000}
\begin{tikzpicture}[y=1cm, x=1cm, yscale=\globalscale,xscale=\globalscale, every node/.append style={scale=\globalscale}, inner sep=0pt, outer sep=0pt]
  \path[fill=grey,opacity=0.5,line cap=butt,line join=miter,line 
  width=0.0cm,miter limit=4.0] (11.0, 27.7) -- (11.0, 21.7) -- (17.0, 21.7) -- 
  (17.0, 27.7) -- cycle;

  \path[draw=black,fill=white,line width=0.0cm] (14.0, 24.7) ellipse (2.0cm and 
  2.0cm);

  \path[draw=black,even odd rule,line cap=butt,line join=miter,line width=0.0cm]
   (5.0, 27.8) -- (5.0, 21.7);

  \path[draw=black,even odd rule,line cap=butt,line join=miter,line 
  width=0.0cm,dash pattern=on 0.1cm off 0.1cm] (2.0, 21.7) -- (8.0, 27.7);

  \path[draw=black,even odd rule,line cap=butt,line join=miter,line 
  width=0.0cm,dash pattern=on 0.1cm off 0.1cm] (2.0, 27.7) -- (8.0, 21.7);

  \path[draw=black,even odd rule,line cap=butt,line join=miter,line width=0.0cm]
   (2.0, 24.7) -- (8.1, 24.7);

  \path[draw=black,line cap=butt,line join=miter,line width=0.0cm,miter 
  limit=4.0] (4.9, 27.7) -- (5.0, 27.8) -- (5.1, 27.7);

  \path[draw=black,line cap=butt,line join=miter,line width=0.0cm,miter 
  limit=4.0] (8.0, 24.8) -- (8.1, 24.7) -- (8.0, 24.6);

  \node at (8.3, 24.3) {\Large$x_1,\ x_2$};

  \node at (4.6, 28.0) {\Large $x_3,\ x_4$};

  \path[draw=black,even odd rule,line cap=butt,line join=miter,line width=0.0cm]
   (14.0, 27.8) -- (14.0, 21.7);

  \path[draw=black,even odd rule,line cap=butt,line join=miter,line width=0.0cm]
   (11.0, 24.7) -- (17.1, 24.7);

  \path[draw=black,line cap=butt,line join=miter,line width=0.0cm,miter 
  limit=4.0] (13.9, 27.7) -- (14.0, 27.8) -- (14.1, 27.7);

  \path[draw=black,line cap=butt,line join=miter,line width=0.0cm,miter 
  limit=4.0] (17.0, 24.8) -- (17.1, 24.7) -- (17.0, 24.6);

  \node at (17.3, 24.3) {\Large $x_1$};

  \node at (13.6, 28.0) {\Large $x_2$};

  \path[draw=black,line cap=butt,line join=miter,line width=0.0cm,miter 
  limit=4.0,dash pattern=on 0.1cm off 0.1cm] (2.7, 26.2) -- (7.3, 26.2);

  \path[fill=black,line width=0.0cm,dash pattern=on 0.1cm off 0.1cm] (12.5, 
  26.2) ellipse (0.1cm and 0.1cm) node [left, xshift=-0.2cm] {\Large $p_+^2$};

  \path[fill=black,line width=0.0cm,dash pattern=on 0.1cm off 0.1cm] (5.8, 26.2)
   ellipse (0.1cm and 0.1cm) node[above, yshift=0.2cm] {\Large $p$};

  \path[draw=black,line cap=butt,line join=miter,line width=0.0cm,miter 
  limit=4.0,dash pattern=on 0.1cm off 0.1cm] (12.5, 26.2) -- (15.5, 26.2);

  \path[fill=black,line width=0.0cm,dash pattern=on 0.1cm off 0.1cm] (15.5, 
  26.2) ellipse (0.1cm and 0.1cm) node [right, xshift=0.2cm] {\Large $p_+^1$};

  \path[fill=black,line width=0.0cm,dash pattern=on 0.1cm off 0.1cm] (14.7, 
  26.2) ellipse (0.1cm and 0.1cm) node [above, yshift=0.2cm] {\Large $p$};

  \path[draw=black,fill=grey,opacity=0.5,line cap=butt,line join=miter,line 
  width=0.0cm,miter limit=4.0] (2.0, 27.4).. controls (4.4, 24.8) and (4.0, 
  24.7) .. (4.0, 24.7).. controls (4.0, 24.7) and (4.4, 24.6) .. (2.0, 22.0);

  \path[draw=black,fill=grey,opacity=0.5,line cap=butt,line join=miter,line 
  width=0.0cm,miter limit=4.0] (8.0, 27.4).. controls (5.6, 24.8) and (6.0, 
  24.7) .. (6.0, 24.7).. controls (6.0, 24.7) and (5.6, 24.6) .. (8.0, 22.0);

  \node at (4.5, 23) {\Huge $\ominus$}; 
  \node at (7, 25.2) {\Huge $\oplus$}; 
  \node at (13.2, 24) {\Huge $\ominus$}; 
  \node at (16, 23) {\Huge $\oplus$};

\end{tikzpicture}
  \caption{hyperboloid type ($+$)}
  \label{fig:cent-h+}
\end{figure}
 The point $p\in \mathbb{R}^8\setminus U_+$ 
lies on the line segment $p_+^1p_+^2$, 
and divides it in the ratio 
$\sqrt{{p_3}^2+{p_4}^2-{p_2}^2+c+\epsilon}-p_1 : 
p_1+\sqrt{{p_3}^2+{p_4}^2-{p_2}^2+c+\epsilon}$ internally. 
 In other words, $p\in\cv{U_+}$. 
 Then it follows $\cv{U_+}=\mathbb{R}^8$. 

  Next, we show $\cv{U_-}=\mathbb{R}^8$. 
 With an arbitrary point $p=(p_1,p_2,\dots,p_8)\in\mathbb{R}^8\setminus U_-$, 
we show $p\in\cv{U_-}$. 
 From the assumption, we have 
\begin{equation*}
  {p_1}^2+{p_2}^2-{p_3}^2-{p_4}^2-c\ge 0. 
\end{equation*}
 Then by setting 
\begin{align*}
  p_-^1&:=\lft(p_1,p_2,\sqrt{{p_1}^2+{p_2}^2-{p_4}^2-c+\epsilon},p_4,
  \dots,p_8\rgt), \\ 
  p_-^2&:=\lft(p_1,p_2,-\sqrt{{p_1}^2+{p_2}^2-{p_4}^2-c+\epsilon},p_4,
  \dots,p_8\rgt),
\end{align*}
for $\epsilon>0$, we have two points $p_-^1,p_-^2\in U_-$ 
(see Figure~\ref{fig:cent-h-}). 
%
%
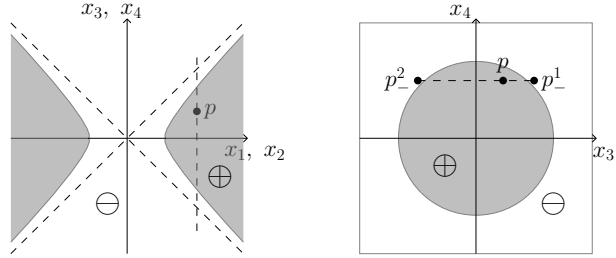
\begin{figure}[htb]
  \centering

\definecolor{grey}{RGB}{128,128,128}

\def \globalscale {0.51000000}
\begin{tikzpicture}[y=1cm, x=1cm, yscale=\globalscale,xscale=\globalscale, every node/.append style={scale=\globalscale}, inner sep=0pt, outer sep=0pt]
  \path[draw=black,opacity=0.5,line cap=butt,line join=miter,line 
  width=0.0cm,miter limit=4.0,dash pattern=on 0.0cm off 0.0cm] (11.0, 27.7) -- 
  (11.0, 21.7) -- (17.0, 21.7) -- (17.0, 27.7) -- cycle;

  \path[draw=black,fill=grey,opacity=0.5,line width=0.0cm] (14.0, 24.7) ellipse 
  (2.0cm and 2.0cm);

  \path[draw=black,even odd rule,line cap=butt,line join=miter,line width=0.0cm]
   (5.0, 27.8) -- (5.0, 21.7);

  \path[draw=black,even odd rule,line cap=butt,line join=miter,line 
  width=0.0cm,dash pattern=on 0.1cm off 0.1cm] (2.0, 21.7) -- (8.0, 27.7);

  \path[draw=black,even odd rule,line cap=butt,line join=miter,line 
  width=0.0cm,dash pattern=on 0.1cm off 0.1cm] (2.0, 27.7) -- (8.0, 21.7);

  \path[draw=black,even odd rule,line cap=butt,line join=miter,line width=0.0cm]
   (2.0, 24.7) -- (8.1, 24.7);

  \path[draw=black,line cap=butt,line join=miter,line width=0.0cm,miter 
  limit=4.0] (4.9, 27.7) -- (5.0, 27.8) -- (5.1, 27.7);

  \path[draw=black,line cap=butt,line join=miter,line width=0.0cm,miter 
  limit=4.0] (8.0, 24.8) -- (8.1, 24.7) -- (8.0, 24.6);

  \node at (8.3, 24.3) {\Large $x_1,\ x_2$};

  \node at (4.6, 28.0) {\Large $x_3,\ x_4$};

  \path[draw=black,even odd rule,line cap=butt,line join=miter,line width=0.0cm]
   (14.0, 27.8) -- (14.0, 21.7);

  \path[draw=black,even odd rule,line cap=butt,line join=miter,line width=0.0cm]
   (11.0, 24.7) -- (17.1, 24.7);

  \path[draw=black,line cap=butt,line join=miter,line width=0.0cm,miter 
  limit=4.0] (13.9, 27.7) -- (14.0, 27.8) -- (14.1, 27.7);

  \path[draw=black,line cap=butt,line join=miter,line width=0.0cm,miter 
  limit=4.0] (17.0, 24.8) -- (17.1, 24.7) -- (17.0, 24.6);

  \node at (17.3, 24.3) {\Large $x_3$};

  \node at (13.6, 28.0) {\Large $x_4$};

  \path[draw=black,line cap=butt,line join=miter,line width=0.0cm,miter 
  limit=4.0,dash pattern=on 0.1cm off 0.1cm] (6.8, 26.8) -- (6.8, 22.2);

  \path[fill=black,line width=0.0cm,dash pattern=on 0.1cm off 0.1cm] (12.5, 
  26.2) ellipse (0.1cm and 0.1cm) node [left, xshift=-0.2cm] {\Large $p_-^2$};

  \path[fill=black,line width=0.0cm,dash pattern=on 0.1cm off 0.1cm] (6.8, 25.4)
   ellipse (0.1cm and 0.1cm) node [right, xshift=0.2cm] {\Large $p$};

  \path[draw=black,line cap=butt,line join=miter,line width=0.0cm,miter 
  limit=4.0,dash pattern=on 0.1cm off 0.1cm] (12.5, 26.2) -- (15.5, 26.2);

  \path[fill=black,line width=0.0cm,dash pattern=on 0.1cm off 0.1cm] (15.5, 
  26.2) ellipse (0.1cm and 0.1cm) node [right, xshift=0.2cm] {\Large $p_-^1$};

  \path[fill=black,line width=0.0cm,dash pattern=on 0.1cm off 0.1cm] (14.7, 
  26.2) ellipse (0.1cm and 0.1cm) node [above, yshift=0.2cm] {\Large $p$};

  \path[draw=black,fill=grey,opacity=0.5,line cap=butt,line join=miter,line 
  width=0.0cm,miter limit=4.0] (2.0, 27.4).. controls (4.4, 24.8) and (4.0, 
  24.7) .. (4.0, 24.7).. controls (4.0, 24.7) and (4.4, 24.6) .. (2.0, 22.0);

  \path[draw=black,fill=grey,opacity=0.5,line cap=butt,line join=miter,line 
  width=0.0cm,miter limit=4.0] (8.0, 27.4).. controls (5.6, 24.8) and (6.0, 
  24.7) .. (6.0, 24.7).. controls (6.0, 24.7) and (5.6, 24.6) .. (8.0, 22.0);

  \node at (4.5, 23) {\Huge $\ominus$}; 
  \node at (7.4, 23.7) {\Huge $\oplus$}; 
  \node at (13.2, 24) {\Huge $\oplus$}; 
  \node at (16, 23) {\Huge $\ominus$};

\end{tikzpicture}
  \caption{hyperboloid type ($-$)}
  \label{fig:cent-h-}
\end{figure}
 The point $p\in \mathbb{R}^8\setminus U_-$ 
lies on the line segment $p_-^1p_-^2$, 
and divides it in the ratio 
$\sqrt{{p_3}^2+{p_4}^2-{p_2}^2+c+\epsilon}-p_3 : 
p_3+\sqrt{{p_3}^2+{p_4}^2-{p_2}^2+c+\epsilon}$ internally. 
 In other words, $p\in\cv{U_-}$. 
 Then it follows $\cv{U_-}=\mathbb{R}^8$.

\noindent
\underline{(2)\ non-central quadrics.}\quad 
In $\mathbb{R}^8$, the function $F_{23}$ 
is reduced to $F_{23}={x_1}^2+{x_2}^2-{x_3}^2-{x_4}^2+x_5$, 
the ``hyperbolic paraboloid'' type. 
 We abuse the same symbol $F_{23}$ for this normal form. 
 The complement of the hypersurface $\{F_{23}=0\}$ is divided 
into the following two path-connected components: 
\begin{equation*}
  U_\pm=\lft\{(x_1,\dots,x_8)\in\mathbb{R}^8\mid 
  {x_1}^2+{x_2}^2-{x_3}^2-{x_4}^2+x_5\gtrless 0\rgt\}. 
\end{equation*}
 We will show that $\cv{U_\pm}=\mathbb{R}^8$. 

  First, we show $\cv{U_+}=\mathbb{R}^8$. 
 With an arbitrary point $p=(p_1,p_2,\dots,p_8)\in\mathbb{R}^8\setminus U_+$, 
we show $p\in\cv{U_+}$. 
 From the assumption, we have 
\begin{equation*}
  {p_1}^2+{p_2}^2-{p_3}^2-{p_4}^2+p_5\le 0. 
\end{equation*}
Then by setting 
\begin{align*}
  p_+^1&:=\lft(\sqrt{{p_3}^2+{p_4}^2-{p_2}^2-p_5+\epsilon},p_2,p_3,p_4,\dots,
  p_8\rgt), \\ 
  p_+^2&:=\lft(-\sqrt{{p_3}^2+{p_4}^2-{p_2}^2-p_5+\epsilon},p_2,p_3,p_4,\dots, 
  p_8\rgt),
\end{align*}
with $\epsilon>0$, we have two points $p_+^1,p_+^2\in U_+$ 
(see Figure~\ref{fig:non-central}). 
%
%
\begin{figure}[htb]
  \centering

\definecolor{grey}{RGB}{128,128,128}
\definecolor{cb3b3b3}{RGB}{179,179,179}

\def \globalscale {0.5000000}
\begin{tikzpicture}[y=1cm, x=1cm, yscale=\globalscale,xscale=\globalscale, every node/.append style={scale=\globalscale}, inner sep=0pt, outer sep=0pt]
  \begin{scope}[shift={(0.0, 1.6)}]
    \path[fill=grey,opacity=0.5,line width=0.0cm,dash pattern=on 0.3cm off 0.1cm
   on 0.0cm off 0.1cm,rounded corners=0.0cm] (16.5, 27.2) rectangle (22.5, 21.2);

    \path[draw=black,fill=white,line cap=butt,line join=miter,line 
  width=0.0cm,miter limit=4.0] (22.2, 27.2).. controls (19.6, 24.8) and (19.5, 
  25.2) .. (19.5, 25.2).. controls (19.5, 25.2) and (19.4, 24.8) .. (16.8, 27.2);

    \path[draw=black,fill=white,line cap=butt,line join=miter,line 
  width=0.0cm,miter limit=4.0] (22.2, 21.2).. controls (19.6, 23.6) and (19.5, 
  23.2) .. (19.5, 23.2).. controls (19.5, 23.2) and (19.4, 23.6) .. (16.8, 21.2);

    \path[fill=grey,opacity=0.5,line cap=butt,line join=miter,line 
  width=0.0cm,miter limit=4.0] (25.5, 27.2) -- (25.5, 21.2) -- (31.5, 21.2) -- 
  (31.5, 27.2) -- cycle;

    \path[draw=black,fill=white,line width=0.0cm] (28.5, 24.2) ellipse (2.0cm 
  and 2.0cm);

    \path[draw=black,even odd rule,line cap=butt,line join=miter,line 
  width=0.0cm] (3.5, 27.3) -- (3.5, 21.2);

    \path[draw=black,even odd rule,line cap=butt,line join=miter,line 
  width=0.0cm,dash pattern=on 0.1cm off 0.1cm] (0.5, 21.2) -- (6.5, 27.2);

    \path[draw=black,even odd rule,line cap=butt,line join=miter,line 
  width=0.0cm,dash pattern=on 0.1cm off 0.1cm] (0.5, 27.2) -- (6.5, 21.2);

    \path[draw=black,even odd rule,line cap=butt,line join=miter,line 
  width=0.0cm] (0.5, 24.2) -- (6.6, 24.2);

    \path[draw=black,line cap=butt,line join=miter,line width=0.0cm,miter 
  limit=4.0] (3.4, 27.2) -- (3.5, 27.3) -- (3.6, 27.2);

    \path[draw=black,line cap=butt,line join=miter,line width=0.0cm,miter 
  limit=4.0] (6.5, 24.3) -- (6.6, 24.2) -- (6.5, 24.1);

    \node at (6.8, 23.8) {\Large $x_1,x_2$};

    \node at (3.1, 27.5) {\Large $x_3,x_4$};

    \path[draw=black,even odd rule,line cap=butt,line join=miter,line 
  width=0.0cm] (28.5, 27.3) -- (28.5, 21.2);

    \path[draw=black,even odd rule,line cap=butt,line join=miter,line 
  width=0.0cm] (25.5, 24.2) -- (31.6, 24.2);

    \path[draw=black,line cap=butt,line join=miter,line width=0.0cm,miter 
  limit=4.0] (28.4, 27.2) -- (28.5, 27.3) -- (28.6, 27.2);

    \path[draw=black,line cap=butt,line join=miter,line width=0.0cm,miter 
  limit=4.0] (31.5, 24.3) -- (31.6, 24.2) -- (31.5, 24.1);

    \node at (31.8, 23.8)  {\Large $x_1$};

    \node at (28.1, 27.5) {\Large $x_2$};

    \path[draw=black,line cap=butt,line join=miter,line width=0.0cm,miter 
  limit=4.0,dash pattern=on 0.1cm off 0.1cm] (1.2, 25.7) -- (5.8, 25.7);

    \path[fill=black,line width=0.0cm,dash pattern=on 0.1cm off 0.1cm] (27.0, 
  25.7) ellipse (0.1cm and 0.1cm) node [left, xshift=-0.2cm] {\Large $p_+^2$};

    \path[fill=black,line width=0.0cm,dash pattern=on 0.1cm off 0.1cm] (4.3, 
  25.7) ellipse (0.1cm and 0.1cm) node [above, yshift=0.2cm] {\Large $p$};

    \path[draw=black,line cap=butt,line join=miter,line width=0.0cm,miter 
  limit=4.0,dash pattern=on 0.1cm off 0.1cm] (27.0, 25.7) -- (30.0, 25.7);

    \path[fill=black,line width=0.0cm,dash pattern=on 0.1cm off 0.1cm] (30.0, 
  25.7) ellipse (0.1cm and 0.1cm) node [right, xshift=0.2cm] {\Large $p_+^1$};

    \path[fill=black,line width=0.0cm,dash pattern=on 0.1cm off 0.1cm] (29.2, 
  25.7) ellipse (0.1cm and 0.1cm) node [above, yshift=0.2cm] {\Large $p$};

    \path[draw=black,fill=grey,opacity=0.5,line cap=butt,line join=miter,line 
  width=0.0cm,miter limit=4.0] (0.5, 26.9).. controls (2.9, 24.3) and (2.5, 
  24.2) .. (2.5, 24.2).. controls (2.5, 24.2) and (2.9, 24.1) .. (0.5, 21.5);

    \path[draw=black,fill=grey,opacity=0.5,line cap=butt,line join=miter,line 
  width=0.0cm,miter limit=4.0] (6.5, 26.9).. controls (4.1, 24.3) and (4.5, 
  24.2) .. (4.5, 24.2).. controls (4.5, 24.2) and (4.1, 24.1) .. (6.5, 21.5);

    \path[fill=cb3b3b3,even odd rule,draw opacity=0.0,line cap=butt,line 
  join=miter,line width=0.0cm] (14.5, 27.2) -- (11.5, 24.2) -- (14.5, 21.2);

    \path[fill=cb3b3b3,even odd rule,draw opacity=0.0,line cap=butt,line 
  join=miter,line width=0.0cm] (8.5, 27.2) -- (11.5, 24.2) -- (8.5, 21.2);

    \path[draw=black,even odd rule,line cap=butt,line join=miter,line 
  width=0.0cm] (11.5, 27.3) -- (11.5, 21.2);

    \path[draw=black,even odd rule,line cap=butt,line join=miter,line 
  width=0.0cm] (8.5, 21.2) -- (14.5, 27.2);

    \path[draw=black,even odd rule,line cap=butt,line join=miter,line 
  width=0.0cm] (8.5, 27.2) -- (14.5, 21.2);

    \path[draw=black,even odd rule,line cap=butt,line join=miter,line 
  width=0.0cm] (8.5, 24.2) -- (14.6, 24.2);

    \path[draw=black,line cap=butt,line join=miter,line width=0.0cm,miter 
  limit=4.0] (11.4, 27.2) -- (11.5, 27.3) -- (11.6, 27.2);

    \path[draw=black,line cap=butt,line join=miter,line width=0.0cm,miter 
  limit=4.0] (14.5, 24.3) -- (14.6, 24.2) -- (14.5, 24.1);

    \node at (14.8, 23.8) {\Large $x_1,x_2$};

    \node at (11.1, 27.5) {\Large $x_3,x_4$};

    \path[draw=black,line cap=butt,line join=miter,line width=0.0cm,miter 
  limit=4.0,dash pattern=on 0.1cm off 0.1cm] (9.7, 25.7) -- (13.3, 25.7);

    \path[fill=black,line width=0.0cm,dash pattern=on 0.1cm off 0.1cm] (12.2, 
  25.7) ellipse (0.1cm and 0.1cm) node [above, yshift=0.2cm] {\Large $p$};

    \path[draw=black,even odd rule,line cap=butt,line join=miter,line 
  width=0.0cm] (22.6, 24.2) -- (16.5, 24.2);

    \path[draw=black,even odd rule,line cap=butt,line join=miter,line 
  width=0.0cm,dash pattern=on 0.1cm off 0.1cm] (16.5, 27.2) -- (22.5, 21.2);

    \path[draw=black,even odd rule,line cap=butt,line join=miter,line 
  width=0.0cm,dash pattern=on 0.1cm off 0.1cm] (22.5, 27.2) -- (16.5, 21.2);

    \path[draw=black,even odd rule,line cap=butt,line join=miter,line 
  width=0.0cm] (19.5, 27.3) -- (19.5, 21.2);

    \path[draw=black,line cap=butt,line join=miter,line width=0.0cm,miter 
  limit=4.0] (22.5, 24.3) -- (22.6, 24.2) -- (22.5, 24.1);

    \path[draw=black,line cap=butt,line join=miter,line width=0.0cm,miter 
  limit=4.0] (19.6, 27.2) -- (19.5, 27.3) -- (19.4, 27.2);

    \node at (19.1, 27.5) {\Large $x_3,x_4$};

    \node at (22.8, 23.8) {\Large $x_1,x_2$};

    \path[draw=black,line cap=butt,line join=miter,line width=0.0cm,miter 
  limit=4.0,dash pattern=on 0.1cm off 0.1cm] (18.1, 25.8) -- (20.8, 25.8);

    \path[fill=black,line width=0.0cm,dash pattern=on 0.1cm off 0.1cm,rotate 
  around={-90.0:(0.0, 29.7)}] (3.9, 49.8) ellipse (0.1cm and 0.1cm) node [above, yshift=0.2cm] {\Large $p$};

    \path[draw=black,fill=black,line cap=butt,line join=miter,line 
  width=0.0cm,miter limit=4.0,dash pattern=on 0.3cm off 0.1cm on 0.0cm off 
  0.1cm] (24.0, 27.2) -- (24.0, 21.2);

  \end{scope}

  \node at (3, 24) {\Huge $\ominus$}; 
  \node at (5.5, 25) {\Huge $\oplus$}; 
  \node at (11, 24) {\Huge $\ominus$}; 
  \node at (13.5, 25) {\Huge $\oplus$}; 
  \node at (19, 23.7) {\Huge $\ominus$}; 
  \node at (21.5, 25) {\Huge $\oplus$}; 
  \node at (27.6, 25) {\Huge $\ominus$}; 
  \node at (30.3, 23.7) {\Huge $\oplus$}; 

  \node at (3.4, 22.2) {\LARGE $p_5<0$}; 
  \node at (11.4, 22.2) {\LARGE $p_5=0$}; 
  \node at (19.4, 22.2) {\LARGE $p_5>0$}; 

\end{tikzpicture}
  \caption{non-central case}
  \label{fig:non-central}
\end{figure}
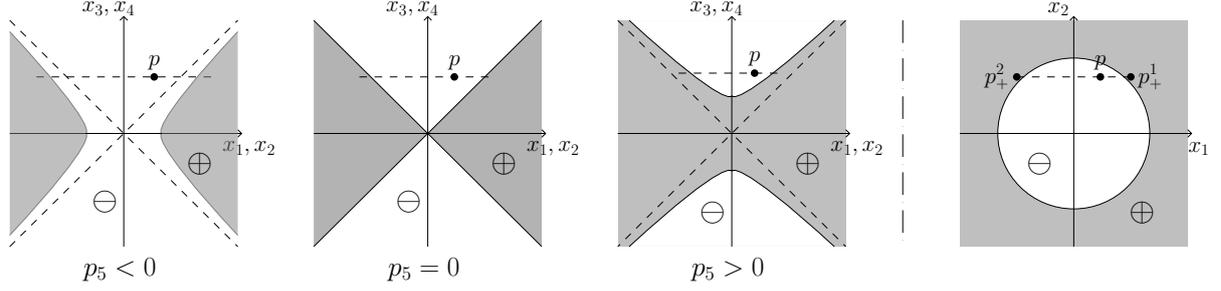 
 The point $p\in \mathbb{R}^8\setminus U_+$ 
lies on the line segment $p_+^1p_+^2$, 
and divides it in the ratio 
$\sqrt{{p_3}^2+{p_4}^2-{p_2}^2-p_5+\epsilon}-p_1 : 
p_1+\sqrt{{p_3}^2+{p_4}^2-{p_2}^2-p_5+\epsilon}$ internally. 
 In other words, $p\in\cv{U_+}$. 
 Then it follows $\cv{U_+}=\mathbb{R}^8$. 

  We can show $\cv{U_-}=\mathbb{R}^8$ in a similar way. 
 With an arbitrary point $p=(p_1,p_2,\dots,p_8)\in\mathbb{R}^8\setminus U_-$, 
we show $p\in\cv{U_-}$. 
 From the assumption, we have 
\begin{equation*}
  {p_1}^2+{p_2}^2-{p_3}^2-{p_4}^2+p_5\ge 0. 
\end{equation*}
Then by setting 
\begin{align*}
  p_-^1&:=\lft(p_1,p_2,\sqrt{{p_1}^2+{p_2}^2-{p_4}^2+p_5+\epsilon},p_4,\dots,
  p_8\rgt), \\ 
  p_-^2&:=\lft(p_1,p_2,-\sqrt{{p_1}^2+{p_2}^2-{p_4}^2+p_5+\epsilon},p_4,\dots, 
  p_8\rgt),
\end{align*}
with $\epsilon>0$, we have two points $p_-^1,p_-^2\in U_-$. 
 The point $p\in \mathbb{R}^8\setminus U_-$ 
lies on the line segment $\overline{p_-^1p_-^2}$, 
and divides it in the ratio 
$\sqrt{{p_1}^2+{p_2}^2-{p_4}^2+p_5+\epsilon}-p_1 : 
p_1+\sqrt{{p_1}^2+{p_2}^2-{p_4}^2+p_5+\epsilon}$ internally. 
 In other words, $p\in\cv{U_-}$. 
 Then it follows $\cv{U_-}=\mathbb{R}^8$. 

  Then we have proved that the differential relation 
$\tilde{\mathcal{R}}\subset\ojet{X_1}$ is ample in Case~2. 

  This completes the proof of Proposition~\ref{prop:essence_(3,5)}.
\end{proof}

\subsection{Proof of Theorem~\ref{thm:key_(3,5)}}\label{sec:proof_(3,5)}
  In this subsection, we show Theorem~\ref{thm:key_(3,5)}. 
 For the proof, 
we apply Gromov's convex integration method, Theorem~\ref{thm:HP_apl_dr}, 
and the \textit{h}-principle for differential sections, 
Theorem~\ref{thm:HP_D-secs}. 
 By Proposition~\ref{prop:essence_(3,5)}, 
we establish Theorem~\ref{thm:key_(3,5)}. 

\begin{proof}[Proof of Theorem~\textup{\ref{thm:key_(3,5)}}]
 Let $M$ be a $5$-dimensional manifold. 
 Suppose that there exits an almost $(3,5)$-distribution 
$(\D,\{\omega_1,\omega_2\})\in\almtf$. 
 Let $\alpha_i$, $i=1,2,3$, denote a local coframing of $\D$. 
 We prove the theorem in the following two steps. 

  First, we apply Proposition~\ref{prop:essence_(3,5)} and 
Gromov's \textit{h}-principle for ample differential relations 
(Theorem~\ref{thm:HP_apl_dr}). 
 Let the vector bundle $X_1:=\bigoplus^{2}T^\ast M$ over $M$ 
be as in Section~\ref{sec:key_(3,5)}. 
 Now, we take the differential relation 
$\tilde{\mathcal{R}}\subset\ojet{X_1}$ to be considered as 
in Equation~\eqref{eq:def_dr_35}. 
 Then, from Proposition~\ref{prop:essence_(3,5)}, 
$\tilde{\mathcal{R}}\subset\ojet{X_1}$ is ample. 
 This implies, according to Gromov's \textit{h}-principle 
for ample differential relations (Theorem~\ref{thm:HP_apl_dr}), 
that the differential relation $\tilde{\mathcal{R}}\subset\ojet{X_1}$ 
satisfies the \textit{h}-principle and the one-parametric \textit{h}-principle. 

  Next, we discuss the \textit{h}-principles 
for first order linear differential operators. 
 For the vector bundle $X_1$ above, let the vector bundle 
$Z_1:=\bigoplus^{2}\lft(T^\ast M\oplus\lft(\bigwedge^2T^\ast M\rgt)\rgt)$ 
over $M$ be as in Section~\ref{sec:key_(3,5)}, 
and the linear differential operator 
$\mathcal{F}_1\colon\sect{X_1}\to\sect{Z_1}$ 
be as in Equation~\eqref{eq:def_f_(3,5)}. 
 Note that the symbol $F_1\colon X_1^{(1)}\to Z_1$ of $\mathcal{F}_1$ 
is fiberwise epimorphic. 
 Recall that the differential relation $\tilde{\mathcal{R}}\subset\ojet{X_1}$ 
is defined as 
$\tilde{\mathcal{R}}=\ojet{X_1}\setminus\tilde{\Sigma}
=\ojet{X_1}\setminus {F_1}^{-1}(\tilde{S})$ 
for the singular locus $\tilde{S}\subset Z_1$ for the almost $(3,5)$ condition 
defined in Equation~\eqref{eq:def_s_(3,5)}. 
 As proved in the previous paragraph, 
the differential relation $\tilde{\mathcal{R}}\subset\ojet{X_1}$
satisfies the \textit{h}-principle and the one-parametric \textit{h}-principle. 
 Then, from the \textit{h}-principle 
for first order linear differential operators (Theorem~\ref{thm:HP_D-secs}), 
we obtain that the \textit{h}-principle 
and the one-parametric \textit{h}-principle hold for the inclusion 
\begin{equation*}
  \secc_{\mathcal{F}_1}\lft(Z_1\setminus \tilde{S}\rgt)
  \hookrightarrow 
  \sect{Z_1\setminus \tilde{S}}, 
\end{equation*}
where
\begin{align*}
  \secc_{\mathcal{F}_1}\lft(Z_1\setminus \tilde{S}\rgt)
  :=
  \lft\{s\in \sect{Z_1\setminus \tilde{S}}
  \;\lft|\; 
  \begin{aligned}
    &s=\mathcal{F}_1(t)\\ 
    &\text{for some}\ t\in\sect{X_1}
  \end{aligned}
  \rgt.\rgt\}
\end{align*}
is the space of $\mathcal{F}_1$-sections. 

  The facts established above, namely, 
the \textit{h}-principle and the one-parametric \textit{h}-principle, 
imply Theorem~\ref{thm:key_(3,5)}. 
 In fact, the set $\sect{Z_1\setminus \tilde{S}}$ 
corresponds to the set of coframings of the almost $(3,5)$-distributions, 
and the set $\secc_{\mathcal{F}_1}\lft(Z_1\setminus \tilde{S}\rgt)$ 
corresponds to the set of coframings 
of the genuine $(3,5)$-distributions. 
 In other words, $\sect{Z_1\setminus \tilde{S}}$ corresponds to $\almtf$, 
and $\secc_{\mathcal{F}_1}\lft(Z_1\setminus \tilde{S}\rgt)$ 
corresponds to $\Omega_{(3,5)}(M)$ 
or $\bar{\Omega}_{(3,5)}(M)\subset\almtf$. 
 Then the \textit{h}-principle implies the claim~(1) 
of Theorem~\ref{thm:key_(3,5)}. 
 Indeed, if there exists an almost $(3,5)$-distribution, 
the \textit{h}-principle guarantees the existence 
of a genuine $(3,5)$-distribution homotopic to the almost $(3,5)$-distribution. 
 The converse is clear. 
 On the other hand, the one-parametric \textit{h}-principle 
implies the claim~(2) of Theorem~\ref{thm:key_(3,5)}. 
 Indeed, the one-parametric \textit{h}-principle guarantees 
that a path in $\sect{Z_1\setminus \tilde{S}}$ connecting two points 
in $\secc_{\mathcal{F}_1}\lft(Z_1\setminus \tilde{S}\rgt)$ can be deformed, 
keeping both ends, 
so that the whole path lies 
in $\secc_{\mathcal{F}_1}\lft(Z_1\setminus \tilde{S}\rgt)$. 
 This means a deformation of the path of almost $(3,5)$-distributions. 

  This completes the proof of Theorem~\ref{thm:key_(3,5)}. 
\end{proof}

\section{Key theorem and corresponding differential relation}
\label{sec:diff-rel}
  We now turn to the proofs of the main theorems in this paper. 
 First, we reformulate the problems in terms of the \emph{h}-principle 
in Section~\ref{sec:keythm}. 
 Then we determine the differential relation and the differential operator 
to be considered in Section~\ref{sec:def-diffrel}. 
\subsection{Reformulation from the viewpoint of the \textit{h}-principles}
\label{sec:keythm}
  We reformulate the problems from the viewpoint of the \textit{h}-principle 
into Theorem~\ref{thm:key}. 
 The structure dealt in Theorems~\ref{thma}, \ref{thmb}, and~\ref{thmc}  
is the Cartan $(2,3,5)$-distributions. 
 For such structures, the formal structures are the almost Cartan structures 
(see Section~\ref{sec:distr} for definitions of such distributions). 
 In order to describe the reformulation, 
we introduce the sets of such distributions. 
 Let $M$ be a possibly closed $5$-dimensional manifold. 
 Let $\Omega_{(2,3,5)}(M)$ denote 
the set of the genuine Cartan $(2,3,5)$-distributions on $M$, 
and $\almcrtn$ the set of almost Cartan structure on $M$. 
 As we mentioned after Definition~\ref{def:alm_cartan}, 
associated with any genuine Cartan $(2,3,5)$-distribution, 
there exists an almost Cartan structure. 
 Then setting $\bar{\Omega}_{(2,3,5)}(M)$ 
as the set of almost Cartan structures 
that are associated with genuine Cartan $(2,3,5)$-distributions, 
we have $\bar{\Omega}_{(2,3,5)}(M)\subset\almcrtn$. 
  Using the notion, the key theorem is formulated as follows. 
%
%
\begin{thrm}\label{thm:key}
Let $M$ be a possibly closed $5$-dimensional manifold. 

\noindent\textup{(1)}\  
 Let $(\D\subset\mathcal{E},\{\omega_1,\omega_2;\omega_3\})\in\almcrtn$ 
be an almost Cartan structure on $M$. 
 Then $(\D\subset,\{\omega_1,\omega_2;\omega_3\})$ is homotopic in $\almcrtn$ 
to an almost Cartan structure associated with 
a genuine Cartan $(2,3,5)$-distribution 
$\mathcal{E}\in\Omega_{(2,3,5)}(M)$ on $M$.


\noindent\textup{(2)}\  
 Let $\D_0,\D_1\in\Omega_{(2,3,5)}(M)$
be genuine Cartan $(2,3,5)$-distributions on $M$, 
and let $(\D_t\subset\mathcal{E}_t,\{\omega_1^t,\omega_2^t;\omega_3^t\})
\in\almcrtn$, $t\in[0,1]$, 
be a path of almost Cartan structures in $\almcrtn$ 
between the associated almost Cartan structures of $\D_0$ and $\D_1$. 
 Then $(\D_t\subset\mathcal{E}_t,\{\omega_1^t,\omega_2^t;\omega_3^t\})$ 
can be deformed in $\almcrtn$ 
to a path $(\tilde{\D}_t\subset{\tilde{\D}_t}^2,
\{[d\tilde{\alpha}_1^t],[d\tilde{\alpha}_2^t];[d\tilde{\alpha}_3^t]\})
\in\bar{\Omega}_{(2,3,5)}(M)$ of almost Cartan structures 
associated with genuine Cartan $(2,3,5)$-distributions 
$\tilde{\D}_t\in\Omega_{(2,3,5)}(M)$, 
while keeping the endpoints fixed: 
\begin{align*}
  (\D_i\subset\mathcal{E}_i,\{\omega_1^i,\omega_2^i;\omega_3^i\}) 
  &=(\D_i\subset{\D_i}^2,\{[d\alpha_1^i],[d\alpha_2^i];[d\alpha_3^i]\}) \\
  &=(\tilde{\D}_i\subset{\tilde{\D}_i}^2,
    \{[d\tilde{\alpha}_1^i],[d\tilde{\alpha}_2^i];[d\tilde{\alpha}_3^i]\}),
    \quad i=0,1.
\end{align*}
\end{thrm} 
\noindent
 This theorem is proved in the next section 
from the viewpoint of the \textit{h}-principles 
introduced in Section~\ref{sec:h-prin}. 
 Theorem~\ref{thma} and Theorem~\ref{thmc} follow this theorem 
(see Section~\ref{sec:pf_thms}).

\subsection{Differential relation for the problem}\label{sec:def-diffrel}
  In order to discuss the problem 
from the viewpoint of the \textit{h}-principle, 
we determine the relevant differential relation in this subsection. 

  In order to deal with tangent distributions, 
we recognize them by their coframings. 
 In addition, an almost Cartan structure is defined as a certain tuple 
consisting of a sequence of distributions and three $2$-forms. 
 In order to describe them, we first introduce two vector bundles 
and a first-order linear differential operator. 
 Let $M$ be a $5$-dimensional manifold, 
and $\D\subset TM$ a tangent distribution of rank~$2$ on $M$. 
 In other words it is of corank~$3$. 
 Then it is locally regarded as a kernel of three $1$-forms. 
 For the description, let $X\to M$ be the vector bundle over $M$ defined as
\begin{equation}\label{eq:Xbdl}
    X:=\bigoplus^3 T^\ast M\to M. 
\end{equation}
 Then a coframing of the distribution $\D$ is locally regarded as a section 
of this bundle $X$. 
 In addition, we recognize almost Cartan structures in a similar manner. 
 From Definition~\ref{def:alm_cartan}, an almost Cartan structure is defined 
as a tuple that consists of a sequence of tangent distributions 
of rank~$2$ and $3$, and three $2$-forms. 
 For the description, let $Z_1\to M$ be a vector bundle over $M$ defined as
\begin{equation}\label{eq:Zbdl}
  Z:=\bigoplus^3\lft( T^\ast M\oplus\bigwedge^2 T^\ast M\rgt)\to M. 
\end{equation}
 Then a tuple consists of a sequence of distributions and three $2$-forms 
is regarded as a section of this bundle $Z$. 
 Next we define the first-order linear differential operator 
connecting the two bundles introduced above. 
 Let $\mathcal{F}\colon\sect{X}\to\sect{Z}$ 
be the first order linear differential operator defined as  
\begin{equation}\label{eq:def_f}
  \mathcal{F}\colon
  (\alpha_1,\ \alpha_2,\ \alpha_3)\mapsto 
  (\alpha_1,\ \alpha_2,\ \alpha_3,\ d\alpha_1,\ d\alpha_2,\ d\alpha_3). 
\end{equation}
 We remark that the symbol $F\colon X^{(1)}\to Z$ 
of $\mathcal{F}$ is fiberwise epimorphic. 
 Then we can apply Theorem~\ref{thm:HP_D-secs} to this operator. 

  Using these descriptions, we formulate the problem as follows. 
 The hypotheses of Theorem~\ref{thm:key} are concerned 
with formal structures, that is, almost Cartan structures. 
 Then we find the singular locus in $Z$, and by Theorem~\ref{thm:HP_D-secs}, 
we determine the differential relation to consider 
in the $1$-jet space $\ojet{X}$.

  Next, we define the differential relation to be considered. 
 It is finally defined in the source side 
of $\mathcal{F}\colon \sect{X}\to\sect{Z}$. 
 The differential relation $\mathcal{R}\subset\ojet{X}$ 
is defined as the complement of the singularity $\Sigma\subset\ojet{X}$ 
that we define bellow. 
 Recall that the Cartan $(2,3,5)$-distribution is defined 
as the tangent distribution $\mathcal{D}$ of rank~$2$ 
on a $5$-dimensional manifold $M$ 
with a certain derived flag $\mathcal{D}\subset \D^2\subset TM$ of length~$2$. 
 Then the singularity $\Sigma\subset\ojet{X}$ is defined as a stratified subset.
 We first define the singular locus $S$ for the almost Cartan condition 
(see Definition~\ref{def:alm_cartan}) in the target side $Z$ 
of $\mathcal{F}\colon\sect{X}\to\sect{Z}$. 
 Then we pull it back by the symbol $F\colon\ojet{X}\to Z$ of $\mathcal{F}$ 
to $\Sigma:=F^{-1}(S)\subset\ojet{X}$. 

  Now, we define the singular locus for the almost Cartan condition in $Z$, 
the target side of $\mathcal{F}$. 
 The singularity $S=S_1\cup S_2\subset Z$ 
with strata $S_1,\ S_2\subset Z$ is defined as follows. 
 The first stratum $S_1\subset Z$ is concerning the first derivation 
$\D^2=[\D,\D]$ from $\D$. 
 Let $S_1\subset Z$ be the subset defined as
\begin{align}\label{eq:def_s1}
  S_1:=\lft\{(\alpha_1,\alpha_2,\ \alpha_3,\ \omega_1,\omega_2,\ \omega_3)_p
  \in\bigoplus^3\lft(T^\ast M\oplus\bigwedge^2T^\ast M\rgt)=Z \rgt. 
  \qquad & \\
  \qquad\qquad \lft|\; 
  \begin{matrix}
    (\alpha_1\wedge\alpha_2\wedge\alpha_3\wedge\omega_1)_p=0,\\
    (\alpha_1\wedge\alpha_2\wedge\alpha_3\wedge\omega_2)_p=0,\\
    (\alpha_1\wedge\alpha_2\wedge\alpha_3\wedge\omega_3)_p=0\mbox{ }
  \end{matrix}
  \rgt\}.& \notag
\end{align}
 In other words, it is where conditions (1)-(i),(ii) in the definition of 
almost Cartan structure do not hold (see Definition~\ref{def:alm_cartan}). 

  The second stratum $S_2\subset Z$ is defined 
in the complement of $S_1$. 
 In $Z\setminus S_1$, at least one 
of $(\alpha_1\wedge\alpha_2\wedge\alpha_3\wedge\omega_i)_p$, $i=1,2,3$, 
is not equal to zero. 
 Then we divide the complement $Z\setminus S_1$ 
according to the number of such $5$-forms as follows. 
 Set 
\begin{equation*}
  U_i:=\{(\alpha_1,\alpha_2,\alpha_3,\omega_1,\omega_2,\omega_3)_p\in Z\mid
  (\alpha_1\wedge\alpha_2\wedge\alpha_3\wedge\omega_i)_p\ne 0\},\quad 
  i=1,2,3. 
\end{equation*}
 The subset where all three $5$-forms are not zero is defined as 
\begin{equation*}
  U^3_{123}:=U_1\cap U_2\cap U_3. 
\end{equation*}
 The subsets where two of those  $5$-forms are not zero are defined as 
\begin{equation*}
  U^2_{12}:=(U_1\cap U_2)\setminus U_3,\quad 
  U^2_{13}:=(U_1\cap U_3)\setminus U_2,\quad 
  U^2_{23}:=(U_2\cap U_3)\setminus U_1. 
\end{equation*}
 The subsets where one of those $5$-forms are not zero are defined as 
\begin{equation*}
  U^1_1:=U_1\setminus (U_2\cap U_3),\quad 
  U^1_2:=U_2\setminus (U_1\cap U_3),\quad 
  U^1_3:=U_3\setminus (U_1\cap U_2). 
\end{equation*}
 Then we have $Z\setminus S_1
=U^3_{123}\cup U^2_{12}\cup U^2_{13}\cup U^2_{23}\cup U^1_1\cup U^1_2\cup U^1_3$. 

 Then, in each $U_i$, $i=1,2,3$, 
we define singular loci $S_2^i$ as follows. 
 For example, we deal with $U_3\subset Z$, 
where $(\alpha_1\wedge\alpha_2\wedge\alpha_3\wedge\omega_3)_p\ne0$. 
 Setting $g^3_i$, $i=1,2$,  as 
\begin{equation*}
  (\alpha_1\wedge\alpha_2\wedge\alpha_3\wedge\omega_i)_p 
  = g^3_i(p)\cdot (\alpha_1\wedge\alpha_2\wedge\alpha_3\wedge\omega_3)_p ,
  \qquad i=1,2, 
\end{equation*}
we define $\tilde{\omega}^3_i$, $i=1,2$, as 
\begin{equation*}
  (\tilde{\omega}^3_i)_p:=(\omega_i)_p-g^3_i(p)\cdot(\omega_3)_p,\quad i=1,2. 
\end{equation*}
 We should remark that 
$(\alpha_1\wedge\alpha_2\wedge\alpha_3\wedge \tilde{\omega}^3_i)_p=0$ 
hold for both $i=1,2$. 
 This implies that the triple $\tilde{\omega}_1, \tilde{\omega}_2,\omega_3$ 
of $2$-forms satisfies the condition~{(1)} of the definition 
of almost Cartan structure (see Definition~\ref{def:alm_cartan}). 
 Now we define the singular locus $S_2^3\subset Z$ as 
\begin{align}\label{eq:def_s23}
   S_2^3:=&\lft\{(\alpha_1,\alpha_2,\alpha_3,\omega_1,\omega_2,\omega_3)_p
  \in U_3\subset\bigoplus^3\lft(T^\ast M\oplus\bigwedge^2T^\ast M\rgt)=Z 
            \rgt. \qquad & \\
          &\qquad \lft.\; \lft\vert \;
    (\alpha_1\wedge\alpha_2\wedge\tilde{\omega}^3_1)_p\quad\text{and}\quad 
      (\alpha_1\wedge\alpha_2\wedge\tilde{\omega}^3_2)_p\phantom{\bigwedge^2}
      \text{are linearly dependent}
  \rgt.\rgt\} \notag
\end{align}
 We define singular loci $S_2^2\subset U_2$ and $S_2^1\subset U_1$ 
in the same way. 

  Now, we define the second stratum $S_2\subset Z$ as follows. 
\begin{equation}\label{eq:def_s2}
  S_2:=
  \begin{cases}
    S_2^1\cap S_2^2\cap S_3^3 &\quad \text{(on $U^3_{123}$),} \\
    S_2^1\cap S_2^2 &\quad \text{(on $U^2_{12}$)}, \\
    S_2^1\cap S_2^3 &\quad \text{(on $U^2_{13}$}), \\
    S_2^2\cap S_2^3 &\quad \text{(on $U^2_{23}$)}, \\
    S_2^1 &\quad \text{(on $U^1_1$)}, \\
    S_2^2 &\quad \text{(on $U^1_2$)}, \\
    S_2^3 &\quad \text{(on $U^1_3$)}. 
  \end{cases}
\end{equation}
%
%
 The stratum $S_2\subset Z\setminus S_1$ corresponds to the locus 
where Condition~(1) in the definition of the almost Cartan structure holds 
but Condition~(2) does not hold (see Definition~\ref{def:alm_cartan}). 

  Combining these strata, we define the singular locus $S\subset Z$ 
for the almost Cartan condition as 
\begin{equation}\label{eq:def_s}
  S:=S_1\cup S_2\subset Z. 
\end{equation}

 Then we define the relevant differential relation in $\ojet{X}$. 
 We take the inverse image of $S=S_1\cup S_2\subset Z$ 
by the symbol $F\colon X^{(1)}\to Z$ 
of the differential operator $\mathcal{F}\colon \sect X\to \sect{Z}$ 
(see Section~\ref{sec:h-prin} for definition). 
 Set 
\begin{equation}\label{eq:def_sgm}
  \Sigma_1:=F^{-1}(S_1),\quad \Sigma_2:=F^{-1}(S_2),\quad 
  \text{and}\quad \Sigma:=F^{-1}(S)\subset \ojet{X}. 
\end{equation}
 Then, since the subset $S\subset Z$ corresponds to 
where the defining conditions of the almost Cartan structure do not hold, 
the subset $\Sigma=F^{-1}(S)\subset \ojet{X}$ is regarded 
as the singularity to be considered. 
 Let $\mathcal{R}\subset \ojet{X}$ denote the complement 
of the singularity $\Sigma$: 
\begin{equation}\label{eq:def_dr}
  \mathcal{R}:=\ojet{X}\setminus\Sigma\subset\ojet{X}. 
\end{equation}
 It is the open differential relation to be considered 
in order to show Theorem~\ref{thm:key}.

\section{Proof of the key theorem}\label{sec:pf_ethm}
  In this section, we show Theorem~\ref{thm:key}, 
the key theorem for the main theorems. 
 As mentioned above, we apply Gromov's convex integration method, 
and the \textit{h}-principle for differential sections. 
 In Section~\ref{sec:ample}, we show that the differential relation 
$\mathcal{R}\subset\ojet{X}$ is ample. 
 Then we apply Gromov's \textit{h}-principle for ample differential relations, 
and the \textit{h}-principle for first order linear differential operators, 
in Section~\ref{sec:proof_cvitgr}. 

\subsection{Ampleness of the differential relation 
$\mathcal{R}\subset\ojet{X}$}\label{sec:ample}
  We show that the differential relation $\mathcal{R}\subset\ojet{X}$ 
obtained in the previous section is ample. 
 The notations introduced in the previous section are used. 
 Let $M$ be a $5$-dimensional manifold, 
$X=\bigoplus^3T^\ast M$ 
and $Z=\bigoplus^3\lft(T^\ast M\oplus\bigwedge^2 T^\ast M\rgt)$ 
the vector bundles over $M$,
$\mathcal{F}\colon\sect{X}\to\sect{Z}$ 
the linear differential operator, 
and $\mathcal{R}=X^{(1)}\setminus\Sigma\subset X^{(1)}$ 
the differential relation. 
%
%
\begin{prop}\label{prop:core}
  The open differential relation $\mathcal{R}\subset X^{(1)}$ is ample. 
\end{prop} 

\begin{proof}
  Recall that the singularity $\mathcal{R}\subset\ojet{X}$ is defined 
as the complement of the singularity $\Sigma\subset\ojet{X}$. 
 And the singularity consists of two strata as $\Sigma=\Sigma_1\cup\Sigma_2$. 
 Then we discuss in two steps. 

  For this kind of discussions, 
it is enough to discuss using local coordinates 
of the base manifold $M$ of the bundles, 
like in Section~\ref{sec:(3,5)-distributions}. 
 Let $(x_1,x_2,\dots,x_5)$ be local coordinates 
of the $5$-dimensional manifold $M$. 
 Then $\{(dx_1)_p,\dots,(dx_5)_p\}$ is a basis of the fiber $T^\ast_p M$. 
 Let $(a_1,\dots,a_5)$ be coordinates on the fiber of $T^\ast M$. 
 Similarly, $\{(dx_i\wedge dx_j)_p\}_{1\le i<j\le 5}$ is a basis of the fiber 
$\lft(\bigwedge^2 T^\ast M\rgt)_p$. 
 Let $(z_{12},\dots,z_{45})$ be coordinates on the fiber 
of $\bigwedge^2 T^\ast M$. 
 Further, let $(a_1,\dots,a_5,y_{11},y_{12},\dots,y_{55})$ be coordinates of 
the fiber $(T^\ast M)^{(1)}_p$, 
where $y_{ij}$ corresponds to $\rd a_i/\rd x_j$. 

\noindent \underline{\textbf{Step~1}}: Singularity $\Sigma_1\subset\ojet{X}$. \\
\indent  Using the local coordinates above, 
we explicitly write down the singularity 
$\Sigma_1\subset\ojet{X}$ as follows. 
 To this end, we first recall the definition of $\Sigma_1$. 
 The first stratum $\Sigma_1$ is defined 
as the inverse image $\Sigma_1=F^{-1}(S_1)$, 
where $F\colon\ojet{X}\to\sect{Z}$ is the symbol 
of the linear differential operator $\mathcal{F}\colon\sect X\to\sect{Z}$ 
defined by Equation~\ref{eq:def_f} concerning exterior derivative 
of differential forms. 
 Note that, by the exterior derivative, 
the coordinates $y_{ij}-y_{ji}$ of $\ojet{T^\ast M}_p$ correspond 
to $z_{ij}$ of $\bigwedge^2T^\ast M$. 
 Recall that the set $S_1\subset Z$ 
is defined by the following condition (see Equation~\eqref{eq:def_s1}): 
\begin{equation}\label{eq:s1_forms}
  (\alpha_1\wedge\alpha_2\wedge\alpha_3\wedge\omega_1)_p=0,\quad
  (\alpha_1\wedge\alpha_2\wedge\alpha_3\wedge\omega_2)_p=0,\quad
  (\alpha_1\wedge\alpha_2\wedge\alpha_3\wedge\omega_3)_p=0, 
\end{equation}
where $\alpha_i\in\sect{T^\ast M}$ are $1$-forms, 
$\omega_i\in\sect{\bigwedge^2 T^\ast M}$ are $2$-forms, and $p\in M$ a point. 
 Then, in order to write down the singularity 
$\Sigma_1=F^{-1}(S_1)\subset \ojet{X}$ by using local coordinates, 
we represent $\alpha_i$ and $\omega_i$ on a fiber over $p\in M$ 
by such coordinates as follows: 
\begin{align*}
  \alpha_i&=\sum_{j=1}^5 a^i_jdx_j 
            =a^i_1dx_1+a^i_2dx_2+\dots+a^i_5dx_5, \quad (i=1,2,3), \\
  \omega_i&=\sum_{1\le j<k\le 5}z^i_{jk}dx_j\wedge dx_k \\
          & =z^i_{12}dx_1\wedge dx_2+z^i_{13}dx_1\wedge dx_3+\dots
            +z^i_{45}dx_{4}\wedge dx_5, \qquad (i=1,2,3). 
\end{align*}
 Following this representation, the $5$-forms 
in Equation~\eqref{eq:s1_forms} are written down as follows. 
 First, we have 
\begin{equation*}
  \alpha_1\wedge\alpha_2\wedge\alpha_3
  =\sum_{1\le i<j<k\le 5}A_{ijk}dx_i\wedge dx_j\wedge dx_k, 
\end{equation*}
where $A_{ijk}$ are minor determinants 
\begin{equation*}
  A_{ijk}:=
  \begin{vmatrix}
    a^1_i & a^1_j & a^1_k\\ a^2_i & a^2_j & a^2_k\\ a^3_i & a^3_j & a^3_k
  \end{vmatrix},\quad 1\le i<j<k\le 5. 
\end{equation*}
 On the other hand, $z^i_{jk}$ is valid for $j<k$. 
 Then the $5$-forms in Equation~\eqref{eq:s1_forms} are 
\begin{align}
  \alpha_1\wedge\alpha_2\wedge\alpha_{3}\wedge\omega_i
  =&\sum_{\substack{\{i,j,k,l,m\}=\{1,2,3,4,5\}\\ i<j<k,\ l<m}}
     \lft(\sigma(i,j,k,l,m)A_{ijk}z^i_{lm}\rgt){\cdot}
     dx_1\wedge dx_2\wedge\dots\wedge dx_5, \\
  =&A_{123}z^i_{45}-A_{124}z^i_{35}+A_{125}z^i
     _{34}+A_{134}z^i_{25}-A_{135}z^i_{24}\notag\\
  &+A_{145}z^i_{23}-A_{234}z^i_{15}+A_{235}z^i_{14}
    -A_{245}z^i_{13}-A_{345}z^i_{12}, \notag \\
   &\qquad i=1,2,3, \notag
\end{align}
where $\sigma(i,j,k,l)$ is the sign of permutation. 
 At last, we obtain a representation of the singularity 
$\Sigma_1\subset\ojet{X}$ by the local coordinates. 
 From the expression in Equation~\eqref{eq:def_s1}, 
the defining condition of $\Sigma_1\subset\ojet{X}$ is 
the following system of equations: 
\begin{equation}\label{eq:descr_s1}
  \lft\{
  \begin{aligned}
   & A_{123}z^1_{45}-A_{124}z^1_{35}+A_{125}z^1
     _{34}+A_{134}z^1_{25}-A_{135}z^1_{24}\\
  &\qquad\quad +A_{145}z^1_{23}-A_{234}z^1_{15}+A_{235}z^1_{14}
    -A_{245}z^1_{13}+A_{345}z^1_{12}=0,  \\
   & A_{123}z^2_{45}-A_{124}z^2_{35}+A_{125}z^2
     _{34}+A_{134}z^2_{25}-A_{135}z^2_{24}\\
  &\qquad\quad +A_{145}z^2_{23}-A_{234}z^2_{15}+A_{235}z^2_{14}
    -A_{245}z^2_{13}+A_{345}z^2_{12}=0,  \\
   & A_{123}z^3_{45}-A_{124}z^3_{35}+A_{125}z^3 
     _{34}+A_{134}z^3_{25}-A_{135}z^3_{24}\\
  &\qquad\quad +A_{145}z^3_{23}-A_{234}z^3_{15}+A_{235}z^3_{14}
    -A_{245}z^3_{13}+A_{345}z^3_{12}=0. 
  \end{aligned}
  \rgt.
\end{equation}
 This is the precise description of $\Sigma_1\subset\ojet{X}$ 
by the local coordinates. 

  In order to show that the differential relation 
$\mathcal{R}_1:=\ojet{X}\setminus\Sigma_1\subset\ojet{X}$ is ample, 
we observe the intersections of $\Sigma_1$ with principal subspaces, 
using local coordinates. 
 To make the discussion simple, we take the principal direction $P$ 
corresponding to $x_1$ in the local coordinates of the base manifold $M$. 
 In other words, in the Principal subspace $P\subset\ojet{X}$, 
only $z^i_{1l}$, $i=1,2,3$, $l=1,\dots, 5$, are variables. 
 Other $z^i_{ml}$, $m\ne 1$, and $a^i_j$ are constant on $P$ 
in a fiber over $p\in M$. 
 Then we observe the intersection $\Sigma_1\cap P$ 
of the singularity $\Sigma_1\subset\ojet{X}$ with the principal subspace $P$. 
 Recall that $\Sigma_1\subset\ojet{X}$ is described 
by the system of equations~\eqref{eq:descr_s1}. 
 In the intersection $\Sigma_1\cap P$, 
only $z^i_{12}, z^i_{13}, z^i_{14}, z^i_{15}$, $i=1,2,3$, are variables. 
 Then $\Sigma_1\cap P$ is described by the following system of linear equations:
\begin{equation}\label{eq:descr_intsct}
  \lft\{
  \begin{aligned}
    & A_{345}z^1_{12}-A_{245}z^1_{13}+A_{235}z^1_{14}-A_{234}z^1_{15}     \\
    &\qquad\quad =-A_{145}z^1_{23}+A_{135}z^1_{24}-A_{134}z^1_{25}
      -A_{125}z^1_{34}+A_{124}z^1_{35}-A_{123}z^1_{45},  \\
    & A_{345}z^2_{12}-A_{245}z^2_{13}+A_{235}z^2_{14}-A_{234}z^2_{15}     \\
    &\qquad\quad =-A_{145}z^2_{23}+A_{135}z^2_{24}-A_{134}z^2_{25}
      -A_{125}z^2_{34}+A_{124}z^2_{35}-A_{123}z^2_{45},  \\
    & A_{345}z^3_{12}-A_{245}z^3_{13}+A_{235}z^3_{14}-A_{234}z^3_{15}     \\
    &\qquad\quad =-A_{145}z^3_{23}+A_{135}z^3_{24}-A_{134}z^3_{25}
      -A_{125}z^3_{34}+A_{124}z^3_{35}-A_{123}z^3_{45}. 
  \end{aligned}
  \rgt. 
\end{equation} 
 Note that this system has 12 variables, and that the right-hand side 
of each three equations is constant. 

  Then we observe the codimension of $\Sigma_1\cap P\subset P$ 
to show that $\ojet{X}\setminus\Sigma\subset\ojet{X}$ is ample. 
 The intersection $\Sigma_1\cap P$ is defined 
by the system of linear equations~\eqref{eq:descr_intsct}. 
 If the system~\eqref{eq:descr_intsct} has no solution, 
$\mathcal{R}_1=\ojet{X}$. 
 Then it is ample. 
 If the system~\eqref{eq:descr_intsct} has solutions, 
the rank of the coefficient matrix of the system 
is the codimension of $\Sigma\cap P\subset P$. 
 The coefficient matrix is the $3\times12$-matrix written down as
\begin{align*}
  &C:= \\
  &\lft(
  \begin{array}{cccccccccccc}
    A_{345} & -A_{245} & A_{235} & -A_{234} & 0&0&0&0&0&0&0&0 \\
    0&0&0&0& A_{345} & -A_{245} & A_{235} & -A_{234} & 0&0&0&0 \\
    0&0&0&0&0&0&0&0& A_{345} & -A_{245} & A_{235} & -A_{234} 
  \end{array}
  \rgt). 
\end{align*}
 When $\rk C=0$, the singularity $\Sigma_1$ 
is either $\ojet{X}$ or~$\emptyset$. 
 Then $\mathcal{R}_1=\ojet{X}\setminus\Sigma_1$ is ample by definition. 
 In order to show that $\mathcal{R}_1\subset\ojet{X}$ is ample, 
we would like to show $\rk C\ne 1$. 
 It is proved as follows. 
 Suppose $\rk C\ne 0$. 
 Then at least one of the rows of $C$ is not $(0\cdots0)$. 
 Since each row has the same entries, all rows are not $(0\cdots0)$. 
 Hence $\rk C=3\ne 1$. 
 Then we conclude that the rank of the coefficient matrix is greater than $1$. 

  Thus, we have proved that the singularity $\Sigma_1$ is thin 
if it is not empty or entire $\ojet{X}$. 
 In other words, the differential relation 
$\mathcal{R}_1=\ojet{X}\setminus\Sigma$ is ample. 

\noindent \underline{\textbf{Step~2}}: 
Singularity $\Sigma=\Sigma_1\cup\Sigma_2\subset\ojet{X}$. \\
\indent  First, we should arrange the cases to consider following Step~1. 
 Setting $\mathcal{R}_2:=\ojet{X}\setminus \Sigma_2$, 
we observe this differential relation. 
 Recall that the strata $\Sigma_2\subset\ojet{X}$ is defined 
after $\Sigma_1$. 
 From Step~1, the intersection $\Sigma_1\cap P$ 
can be either the empty set, the entire $P$, or is thin. 
 In each case, we observe the relation between the ampleness 
of the differential relations $\mathcal{R}$ and that of $\mathcal{R}_2$. \\
(1)\ If $\Sigma_1\cap P=P$, 
then $\Sigma_2\cap P=\emptyset$ and $\Sigma\cap P=P$. 
 Therefore, $\mathcal{R}\cap P=(\ojet{X}\setminus\Sigma)\cap P\subset P$ 
is ample in this case. \\
(2)\  If $\Sigma_1\cap P=\emptyset$, then $\Sigma\cap P=\Sigma_2\cap P$. 
 Therefore, if $\mathcal{R}_2\cap P\subset P$ is ample 
$\mathcal{R}\cap P\subset P$ is ample there. \\
(3)\ If $\Sigma_1\cap P\subset P$ is thin, then $P\setminus\Sigma_1\subset P$ 
is path-connected, and $\cv{P\setminus\Sigma_1}=P$. 
 Then, if $U\subset\mathcal{R}_2\cap P$ is a path-connected component, 
$U\setminus\Sigma_1\subset\mathcal{R}\cap P$ 
is a path-connected component. 
 In addition, $\cv{U}=\cv{U\setminus\Sigma_1}$. 
 Therefore, if $\mathcal{R}_2\cap P\subset P$ is ample, 
$\mathcal{R}\cap P\subset P$ is ample. \\
 In summary, if $\mathcal{R}_2\cap P\subset P$ is ample 
$\mathcal{R}\cap P\subset P$ is ample. 
 In other words, in order to show 
that the differential relation $\mathcal{R}\subset\ojet{X}$ is ample, 
it is sufficient to prove that $\mathcal{R}_2\subset\ojet{X}$ is ample. 

  Now, we show that $\mathcal{R}_2=\ojet{X}\setminus\Sigma_2\subset\ojet{X}$ 
is ample. 
 First, we recall the definition 
of the differential relation $\mathcal{R}_2\subset\ojet{X}$ 
or the singularity  $\Sigma_2\subset\ojet{X}$. 
 The singularity $\Sigma_2$ is defined as $\Sigma_2:=F^{-1}(S_2)$ 
(see Equations~\eqref{eq:def_sgm}). 
 The singular locus $S_2\subset Z$ is defined as combination 
of $S_2^1, S_2^2, S_2^3\subset Z$ (see Equation~\eqref{eq:def_s2}). 
 The important thing is that the definition of $S_2^i\subset Z$ 
is the same as that of $\tilde{S}\subset Z_1$ 
(see Equation~\eqref{eq:def_s_(3,5)} and Equation~\eqref{eq:def_s23}). 
 In other words, both are singular loci for almost $(3,5)$-distributions. 
 Then we can apply the same arguments 
as in Section~\ref{sec:(3,5)-distributions} to each $S_2^i\subset Z$ 
and $\Sigma_2^i=F^{-1}(S_2^i)\subset \ojet{X}$, $i=1,2,3$. 
 As a result, we conclude that each complement 
$\ojet{X}\setminus\Sigma_2^i\subset$ is ample, $i=1,2,3$, 
like Proposition~\ref{prop:essence_(3,5)}. 
 In addition, the singular locus $S_2\subset Z$ is defined 
by the intersections of these $S_2^i\subset Z$, $i=1,2,3$, that is, 
subsets of each $S_2^i$. 
 Therefore the differential relation 
$\mathcal{R}_2=\ojet{X}\setminus\Sigma_2\subset\ojet{X}$ is ample. 

  This completes the proof of Proposition~\ref{prop:core}. 
\end{proof}

\subsection{Proof of Theorem~\ref{thm:key}}\label{sec:proof_cvitgr}
 From the result in the previous subsection, we prove Theorem~\ref{thm:key}. 
 Gromov's convex integration method (Theorem~\ref{thm:HP_apl_dr}), 
and the \textit{h}-principle for first order linear differential operators 
(Theorem~\ref{thm:HP_D-secs}) are applied. 

\begin{proof}[Proof of Theorem~\textup{\ref{thm:key}}]
 Let $M$ be a $5$-dimensional manifold. 
 Suppose that there exists an almost Cartan structure 
$(\D\subset\mathcal{E},\{\omega_1,\omega_2;\omega_3\})\in\almcrtn$ 
(see Definition~\ref{def:alm_cartan}). 
 Let $\alpha_i$, $i=1,2,3$, denote a local coframing of $\D$. 
 We prove the theorem in the following two steps. 

  First, we apply Proposition~\ref{prop:core} and 
Gromov's \textit{h}-principle for ample differential relations 
(Theorem~\ref{thm:HP_apl_dr}). 
 Let the vector bundle $X:=\bigoplus^{3}T^\ast M$ over $M$ 
be as in Equation~\eqref{eq:Xbdl}. 
 Now, we take the differential relation 
$\mathcal{R}\subset\ojet{X}$ to be considered as in Equation~\eqref{eq:def_dr}. 
 Then, from Proposition~\ref{prop:core}, $\mathcal{R}\subset\ojet{X}$ is ample. 
 This implies, according to Gromov's \textit{h}-principle 
for ample differential relations (Theorem~\ref{thm:HP_apl_dr}), 
that the differential relation $\mathcal{R}\subset\ojet{X}$ satisfies 
the \textit{h}-principle and the one-parametric \textit{h}-principle. 

  Next, we discuss the \textit{h}-principles 
for first order linear differential operators. 
 For the vector bundle $X$ above, let the vector bundle 
$Z:=\bigoplus^{3}\lft(T^\ast M\oplus\lft(\bigwedge^2T^\ast M\rgt)\rgt)$ 
over $M$ be as in Equation~\eqref{eq:Zbdl}, 
and the linear differential operator 
$\mathcal{F}\colon\sect{X}\to\sect{Z}$ 
be as in Equation~\eqref{eq:def_f}. 
 Note that the symbol $F\colon X^{(1)}\to Z$ of $\mathcal{F}$ 
is fiberwise epimorphic. 
 Recall that the differential relation $\mathcal{R}\subset\ojet{X}$ is defined 
as $\mathcal{R}=\ojet{X}\setminus\Sigma=\ojet{X}\setminus F^{-1}(S)$ 
for the singular locus $S\subset Z$ for the almost Cartan condition 
defined in Equation~\eqref{eq:def_s}. 
 As proved in the previous paragraph, 
the differential relation $\mathcal{R}\subset\ojet{X}$
satisfies the \textit{h}-principle and the one-parametric \textit{h}-principle. 
 Then, from the \textit{h}-principle 
for first order linear differential operators (Theorem~\ref{thm:HP_D-secs}), 
we obtain that the \textit{h}-principle 
and the one-parametric \textit{h}-principle hold for the inclusion 
\begin{equation*}
  \secc_{\mathcal{F}}\lft(Z\setminus S\rgt)
  \hookrightarrow 
  \sect{Z\setminus S}, 
\end{equation*}
where
\begin{align*}
  \secc_{\mathcal{F}}\lft(Z\setminus S\rgt)
  :=
  \lft\{s\in \sect{Z\setminus S}
  \;\lft|\; 
  \begin{aligned}
    &s=\mathcal{F}(t)\\ 
    &\text{for some}\ t\in\sect{X}
  \end{aligned}
  \rgt.\rgt\}
\end{align*}
is the space of $\mathcal{F}$-sections. 

  The facts established above, namely, 
the \textit{h}-principle and the one-parametric \textit{h}-principle, 
imply Theorem~\ref{thm:key}. 
 In fact, the set $\sect{Z\setminus S}$ 
corresponds to the set of coframings of the almost Cartan structures, 
and the set $\secc_{\mathcal{F}}\lft(Z\setminus S\rgt)$ 
corresponds to the set of coframings 
of the genuine Cartan~$(2,3,5)$-distributions. 
In other words, $\sect{Z\setminus S}$ corresponds to $\almcrtn$, 
and $\secc_{\mathcal{F}}\lft(Z\setminus S\rgt)$ 
corresponds to $\Omega_{(2,3,5)}(M)$ 
or $\bar{\Omega}_{(2,3,5)}(M)\subset\almcrtn$. 
 Then the \textit{h}-principle implies the claim~(1) of Theorem~\ref{thm:key}. 
 Indeed, if there exists an almost Cartan structure, 
the \textit{h}-principle guarantees the existence 
of a genuine Cartan $(2,3,5)$-distribution 
homotopic to the almost Cartan structure. 
 The converse is clear. 
 On the other hand, the one-parametric \textit{h}-principle 
implies the claim~(2) of Theorem~\ref{thm:key}. 
 Indeed, the one-parametric \textit{h}-principle guarantees 
that a path in $\sect{Z\setminus S}$ connecting two points 
in $\secc_{\mathcal{F}}\lft(Z\setminus S\rgt)$ can be deformed, 
keeping both ends, 
so that the whole path lies in $\secc_{\mathcal{F}}\lft(Z\setminus S\rgt)$. 
 This means a deformation of the path of almost Cartan structures. 

  This completes the proof of Theorem~\ref{thm:key}. 
\end{proof}

\subsection{Proofs of main theorems}\label{sec:pf_thms}
 In this subsection, we show Theorem~\ref{thma} and Theorem~\ref{thmc} 
from the key theorem, Theorem~\ref{thm:key}. 
\begin{proof}[Proof of Theorem~\ref{thma}]
  As was noted after Definition~\ref{def:alm_cartan}, 
  for each ``genuine'' Cartan $(2,3,5)$-distribution 
  there exists an associated almost Cartan structure. 
  Then it remains to show the converse. 
  Suppose that there exists an almost Cartan structure 
  $(\D,\{\omega_1,\omega_2;\omega_3\})$ on a $5$-dimensional manifold $M$. 
  Then, by Theorem~\ref{thm:key}, 
  $(\D,\{\omega_1,\omega_2;\omega_3\})\in\almcrtn$ is homotopic in $\almcrtn$ 
  to an almost Cartan structure 
  $(\D,\{[d\alpha_1],[d\alpha_2];[d\alpha_3]\})\in\bar{\Omega}_{(2,3,5)}(M)$ 
  associated with the genuine Cartan $(2,3,5)$-distribution $\D$. 
  That is, there exists a genuine Cartan $(2,3,5)$-distribution 
  $\D\in\Omega_{(2,3,5)}(M)$ on $M$. 
\end{proof}

\begin{proof}[Proof of Theorem~\ref{thmc}]
  Let $\D_0,\ \D_1\in\Omega_{(2,3,5)}(M)$ be genuine $(2,3,5)$-distributions. 
  Suppose that there exists a homotopy 
  $(\D_t,\{\omega_1^t,\omega_2^t;\omega_3^t\})\in\almcrtn$, $t\in[0,1]$, 
  between two almost $(2,3,5)$-distributions 
  $(\D_i,\{\omega_1^i,\omega_2^i;\omega_3^i\})
  =(\D_i,\{[d\alpha_1^i],[d\alpha_2^i];[d\alpha_3^i]\})
  \in\bar{\Omega}_{(2,3,5)}(M)$ 
  associated with the given genuine Cartan $(2,3,5)$-distribution 
  $\D_i$, for $i=0,1$. 
  Then, by Theorem~\ref{thm:key}, there exists a path 
  $(\tilde{\D}_t,\{[d\tilde{\alpha}_1^t],[d\tilde{\alpha}_2^t]\})
  \in\bar{\Omega}_{(2,3,5)}(M)$ 
  of almost Cartan structures associated 
  with genuine Cartan $(2,3,5)$-distributions 
  $\tilde{\D}_t\in\Omega_{(2,3,5)}(M)$, where $\tilde{\D}_i=\D_i$, $i=0,1$. 
  That is, there exists a homotopy 
  $\tilde{\D}_t\in\Omega_{(2,3,5)}(M)$, $t\in[0,1]$, 
  of the genuine Cartan $(2,3,5)$-distributions 
  between $\D_0=\tilde{\D}_0$ and $\D_1=\tilde{\D}_1$.
\end{proof}

%
%

\end{document}